\documentclass[12pt]{amsart}
\usepackage{fullpage}
\usepackage{amsmath,amsthm,amssymb,amscd,float}
\usepackage{subcaption}
\usepackage[shortlabels]{enumitem}
\usepackage{youngtab}
\usepackage{tikz}
\usepackage{xcolor}
\usepackage{amsfonts,hyperref}
\usepackage{tabularx}
\usepackage[normalem]{ulem}
\usetikzlibrary{arrows,matrix}

\newtheorem{theorem}{Theorem}[section]

\newtheorem{corollary}[theorem]{Corollary}
\newtheorem{proposition}[theorem]{Proposition}

\newtheorem{lemma}[theorem]{Lemma}
\newtheorem{remark}{Remark}[section]
\usepackage[english]{babel}
\usepackage{dutchcal}
\allowdisplaybreaks

\usepackage{ulem}

\title[Inequalities and asymptotics for hook numbers]{ Inequalities and asymptotics for hook numbers in two partition functions of MacMahon and Andrews}

\author[R. Barman]{Rupam Barman}
\address{Department of Mathematics, Indian Institute of Technology Guwahati, Assam, India, PIN- 781039}
\email{rupam@iitg.ac.in}

\author[A. Deb]{Archit Deb}
\address{Department of Mathematics, Indian Institute of Technology Guwahati, Assam, India, PIN- 781039}
\email{archit.deb@iitg.ac.in}

\author[S. Ghorai]{Sayantan Ghorai}
\address{Department of Mathematics, Indian Institute of Technology Guwahati, Assam, India, PIN- 781039}
\email{g.sayantan@iitg.ac.in}


\begin{document}

\begin{abstract}
Motivated by recent studies of hook length biases in restricted partitions, we investigate the total number of hooks of fixed length in two partition functions introduced by MacMahon and Andrews. In $1916$, MacMahon proved that the number of partitions of $n$ in which no part appears with multiplicity one equals the number of partitions of $n$ into parts that are even or congruent to $3 \pmod{6}$. For a positive integer $r$, let $A_r(n)$ count the partitions of $n$ in which every part occurring an odd number of times occurs at least $2r+1$ times, and let $B_r(n)$ count the partitions into parts that are even or congruent to $2r+1 \pmod{4r+2}$. Andrews proved that $A_r(n)=B_r(n)$ for all positive integers $r$ and nonnegative integers $n$, thereby recovering MacMahon's identity as the special case $r=1$. We derive generating functions and asymptotic formulas for the total numbers of hooks of lengths $1$, $2$, and $3$ in both partition classes, employing Wright's circle method. These asymptotic formulas yield several hook length bias results. In particular, for sufficiently large $n$, the $B_r$-type partitions exhibit more $1$-hooks than the $A_r$-type partitions, whereas for hook lengths $2$ and $3$, the bias is reversed when $r=1$ and favors the $B_r$-type partitions when $r\geq 2$. We also establish asymptotic inequalities among hook numbers of different lengths within each partition class. Finally, for fixed hook lengths $\ell\in\{1,2,3\}$, we prove that the $B_r$-type partitions exhibit more $\ell$-hooks than the $B_{r+1}$-type partitions for sufficiently large $n$. These results reveal systematic asymptotic patterns in the distribution of small hook lengths across the two partition families.
\end{abstract}

\maketitle

\section{Introduction}\label{p1_sec1}
A partition $\lambda$ of a positive integer $n$ is a finite sequence of non-increasing positive integers $\lambda=(\lambda_1, \lambda_2, \ldots, \lambda_s)$ such that $\lambda_1+\lambda_2+\cdots +\lambda_s=n$. The numbers $\lambda_1, \lambda_2, \ldots, \lambda_s$ are called parts of the partition $\lambda$.
\par 
A \textit{Young diagram} of a partition $(\lambda_1, \lambda_2, \ldots, \lambda_s)$ is a left-justified array of boxes, where the $i$-th row from the top contains $\lambda_i$ boxes. For example, the Young diagram of the partition $(6,3,2,2)$ is shown in Figure \ref{fig-hook-length-1} (left). For each box
$v$ in the Young diagram of a partition $\lambda$, the
\textit{arm length} (resp. the \textit{co-arm length}, the
\textit{leg length}) of $v$, denoted by
$a_\lambda(v)$ (resp. $ca_\lambda(v)$, $l_\lambda(v)$), is
defined as the number of boxes to the right of (resp. to the
left of, below) $v$ in the Young diagram of $\lambda$. For example, see Figure \ref{fig-hook-length-2}. The \textit{hook length}, also known as the \emph{hook number}, of $v$ is defined as
\[
h_\lambda(v)=a_\lambda(v)+l_\lambda(v)+1
.\]
For example, see Figure \ref{fig-hook-length-1} (right) for the hook lengths of each box in the Young diagram of the partition $(6,3,2,2)$. A hook of length $k$ is also called a $k$-hook.

\begin{figure}[h]
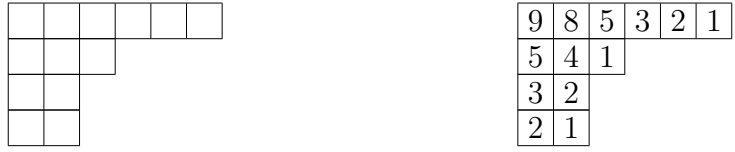

\centering
\begin{minipage}[b]{0.4\textwidth}
\[\young(~~~~~~,~~~,~~,~~)\]
\end{minipage}
\begin{minipage}[b]{0.4\textwidth}
\[\young(985321,541,32,21)\]
\end{minipage}
\caption{The Young diagram of the partition $(6,3,2,2)$ and its hook lengths}\label{fig-hook-length-1}
\end{figure}

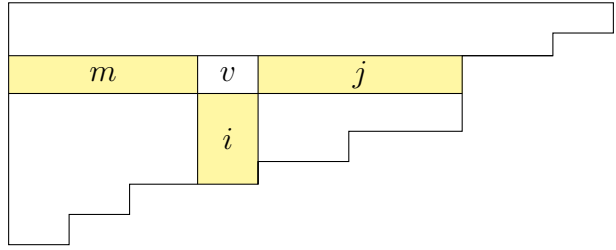
\begin{figure}[htbp]
\centering

\begin{tikzpicture}[x=1cm,y=1cm]

    % Shaded regions
    \fill[yellow!45] (0,2.0) rectangle (2.5,2.5);
    \fill[yellow!45] (3.3,2.0) rectangle (6.0,2.5);
    \fill[yellow!45] (2.5,0.8) rectangle (3.3,2.0);

    % Outer boundary
    \draw
        (0,0) -- (0,3.2)
        -- (8,3.2)
        -- (8,2.8)
        -- (7.2,2.8)
        -- (7.2,2.5)
        -- (6.0,2.5)
        -- (6.0,2.0)
        -- (6.0,1.5)
        -- (4.5,1.5)
        -- (4.5,1.1)
        -- (3.3,1.1)
        -- (3.3,0.8)
        -- (1.6,0.8)
        -- (1.6,0.4)
        -- (0.8,0.4)
        -- (0.8,0)
        -- cycle;

    % Horizontal dividing lines
    \draw (0,2.0) -- (6.0,2.0);
    \draw (0,2.5) -- (7.2,2.5);

    % Vertical boundaries
    \draw (2.5,2.0) -- (2.5,2.5);
    \draw (3.3,2.0) -- (3.3,2.5);
    \draw (3.3,0.8) -- (3.3,2.0);
    \draw (2.5,0.8) -- (2.5,2.0);
    \draw (3.3,0.8) -- (3.3,1.1);

    % Labels
    \node at (1.25,2.25) {$m$};
    \node at (2.9,2.25) {$v$};
    \node at (4.65,2.25) {$j$};
    \node at (2.9,1.4) {$i$};

\end{tikzpicture}
\caption{Arm $a_\lambda(v) = j$, coarm $ca_\lambda(v) = m$ and leg $l_\lambda(v) = i$.}
\label{fig-hook-length-2}
\end{figure}
\par 
The hook lengths of partitions have an important connection with the representation theory of symmetric groups $S_n$ and $GL_n(\mathbb{C})$. The irreducible representations of $S_n$ are in bijection with the partitions of $n$, and if $\rho_{\lambda}$ is the irreducible representation corresponding to a partition $\lambda$, then the dimension of $\rho_{\lambda}$ is a function of the hook numbers of $\lambda$.
Hook lengths also appear in the Seiberg-Witten theory of random partitions, which gives the Nekrasov-Okounkov formula for arbitrary powers of Euler's infinite product in terms of hook numbers, see \cite{Han_2010, Nekrasov_Okounkov} for more details. These connections have led to many interesting studies regarding the asymptotic, combinatorial, and arithmetic properties of hook numbers, particularly those concerning $t$-core partitions and $t$-hooks.
\par 
Recently, several authors have investigated the hook lengths associated with various partition functions. Beyond the classical partition function, significant attention has been given to the study of hook lengths within several restricted partition functions, such as partitions into odd parts, partitions into distinct parts, partitions simultaneously into odd and distinct parts, self-conjugate partitions, and doubled distinct partitions (see, for instance, \cite{Ono_Singh, Ballantine_2023, Yee, Craig3, Craig, Craig2, Griffin, Han_2016, Han_2017, Petreolle}). In a significant contribution, Ballantine et al. \cite{Ballantine_2023} established the existence of specific hook length biases between partitions into odd parts and those into distinct parts. Further studies have explored such hook length biases; see, for example, \cite{Cossaboom, Craig, Kim1}. Recently, the first author and G. Singh studied hook length biases in ordinary and $t$-regular partitions and proposed several open problems; see, for example,  \cite{Singh1Barman, Singh2Barman}. Subsequently, the first author, P. Mahanta and G. Singh \cite{BMS}, as well as He et al. \cite{Bing-He} and Qu et al. \cite{Zang} investigated these open problems. Motivated by the hook length biases for ordinary partitions established in \cite{Singh1Barman}, Cho et al. \cite{Yee} investigated analogous biases for self-conjugate and doubled distinct partitions.

\par
In this article, we consider two restricted partition functions introduced by MacMahon \cite{MacMahon} and their natural generalization due to Andrews \cite{MacMahon_Generalization_by_Andrews}. In \cite[p. 54]{MacMahon}, MacMahon proved the following theorem: 
\begin{theorem}[\cite{MacMahon}]\label{p1_thm1.1}
The number of partitions of $n$ wherein no part appears with
multiplicity one equals the number of partitions of $n$ where all
parts must be even or congruent to $3 \pmod{6}$.
\end{theorem}
MacMahon used a generating function argument to prove Theorem~\ref{p1_thm1.1}. In $1967$, Andrews \cite{MacMahon_Generalization_by_Andrews} stated and proved the following natural generalization of Theorem~\ref{p1_thm1.1}:
\begin{theorem}[\cite{MacMahon_Generalization_by_Andrews}]\label{p1_thm1.2}
Let $A_r(n)$ denote the number of partitions of $n$ of the form $n=\lambda_1+\cdots +\lambda_s$, where $\lambda_i\ge \lambda_{i+1}$, all odd parts are $\ge 2r+1$, and if $\lambda_i-\lambda_{i+1}$ is odd then $\lambda_i - \lambda_{i+1}\ge 2r+1$. Let $B_r(n)$ be the number of partitions of $n$ into parts which are either even or else congruent to $2r+1$ modulo $4r+2$. Then $A_r(n)=B_r(n)$.
\end{theorem}
The case $r=0$ is trivial. The case $r=1$ of Theorem~\ref{p1_thm1.2} reduces to Theorem~\ref{p1_thm1.1} of MacMahon. As an example, when $r=2$ and $n=15$, the nine partitions enumerated by $A_2(15)$ are
$
15, 13+2, 11+4, 11+2+2, 10+5,
9+4+2, 9+2+2+2, 7+2+2+2+2, 5+5+5,$
while the nine partitions enumerated by $B_2(15)$ are
$
15, 10+5, 8+5+2, 6+5+4, 6+5+2+2,
5+5+5, 5+4+4+2, 5+4+2+2+2, 5+2+2+2+2+2.$ As in MacMahon's proof, Andrews' argument is a straightforward application of generating functions. In \cite{Sellers}, Fu and Sellers gave bijective proofs of Theorem~\ref{p1_thm1.1} and Theorem~\ref{p1_thm1.2}.
\par 
In this article, we investigate hook length biases between the two partition families associated with the MacMahon-Andrews identity. Recall that the conjugate of a partition counted by $A_r(n)$ is a partition in which every part occurring an odd number of times occurs at least $2r+1$ times. \emph{Since conjugation preserves hook lengths, we adopt this equivalent characterization of $A_r(n)$ throughout our study of hook numbers.} 
\par Let $a_{\ell}^{(r)}(n)$ and $b_{\ell}^{(r)}(n)$ denote the total numbers of hooks of length $\ell$ in all partitions of $n$ counted by $A_r(n)$ and $B_r(n)$, respectively. Since $A_r(n)=B_r(n)$ for all $n,r\geq 1$, we have, for every $n, r\geq 1$,
\begin{align}\label{eqn-new-13}
nA_r(n)=\sum_{\ell\geq 1}a_{\ell}^{(r)}(n)=\sum_{\ell\geq 1}b_{\ell}^{(r)}(n)=nB_r(n).
\end{align}
In light of \eqref{eqn-new-13}, any inequality between $a_{\ell_0}^{(r)}(n)$ and $b_{\ell_0}^{(r)}(n)$, for a fixed $\ell_0$, naturally leads us to investigate inequalities between $a_{\ell}^{(r)}(n)$ and $b_{\ell}^{(r)}(n)$ for other values of $\ell$. Such inequalities must collectively compensate for the difference at $\ell=\ell_0$, ensuring that the sums over all $\ell$ remain equal. In search of such inequalities, we focus on hook lengths $\ell\in\{1,2,3\}$ and establish asymptotic inequalities both between the two partition families and within each family. We also investigate how the hook numbers associated with $B_r(n)$ vary with $r$.
\par Our approach is based on deriving generating functions for $a_{\ell}^{(r)}(n)$ and $b_{\ell}^{(r)}(n)$, followed by obtaining their asymptotic formulas using Wright's circle method. These formulas yield the hook length bias results stated below.
%%%%%%%%%%%%%%%%%%%%%%%%%%%%%%%%%%%%%%%%%%%%%%%%%%%%%%%%%
\subsection{Main results}
We now present our main results on hook length biases for the partition functions $A_r(n)$ and $B_r(n)$. We first compare the total numbers of hooks of lengths 1, 2, and 3 between the two partition families. We then establish inequalities among hook numbers of different lengths within each family. Finally, we investigate the dependence of the hook numbers for $B_r(n)$ on the parameter $r$. In the following theorem, we show that, for all $r\geq 1$, there are more hooks of length $1$ in all partitions of $n$ counted by $B_r(n)$ than in all partitions of $n$ counted by $A_r(n)$.
\begin{theorem} \label{p1_thm1.3}
For every integer $r\geq 1$, there exists $N_1^{(r)}\in\mathbb{N}$ such that
$b_1^{(r)}(n)>a_1^{(r)}(n)$ for all $n\geq N_1^{(r)}$. Moreover,
$b_1^{(r)}(n)-a_1^{(r)}(n) \rightarrow\infty$ as $n\to\infty$,
and
$$
\lim_{r\to\infty}
\left(\lim_{n\to\infty}
\frac{b_1^{(r)}(n)}{a_1^{(r)}(n)}\right)=1.
$$ 
\end{theorem}
For hook lengths $2$ and $3$, the direction of the bias depends on $r$: when $r=1$, the $A_1$-type partitions have more hooks of these lengths, whereas for $r\geq 2$, the bias favors the $B_r$-type partitions.
\begin{theorem}\label{p1_thm1.4}
\begin{enumerate}[label=(\roman*)]
    \item There exists a positive integer $N_2^{(1)}$ such that for all $n\geq N_2^{(1)}$, we have $b_2^{(1)}(n)<a_2^{(1)}(n)$.
     \item For every integer $r\geq 2$, there exists a positive integer $N_2^{(r)}$ such that for all $n\geq N_2^{(r)}$, we have $b_2^{(r)}(n)>a_2^{(r)}(n)$.
    \item For every integer $r\ge2$, we have $b_2^{(r)}(n)-a_2^{(r)}(n)\to \infty$ as $n \to \infty$. Also, we have $a_2^{(1)}(n)-b_2^{(1)}(n)\to \infty$ as $n\to \infty$.
    \item Furthermore, we have $$\displaystyle \lim_{r\to \infty}\left(\lim_{n\to \infty}\frac{b_2^{(r)}(n)}{a_2^{(r)}(n)}\right)=1.$$
\end{enumerate}
\end{theorem}
\begin{theorem}\label{p1_thm1.5}
\begin{enumerate}[label=(\roman*)]
\item There exists a positive integer $N_3^{(1)}$ such that for all $n\geq N_3^{(1)}$, we have $b_3^{(1)}(n)<a_3^{(1)}(n)$.
    \item For every integer $r\geq 2$, there exists a positive integer $N_3^{(r)}$ such that for all $n\geq N_3^{(r)}$, we have $b_3^{(r)}(n)>a_3^{(r)}(n)$.
    \item For every integer $r\ge2$, we have $b_3^{(r)}(n)-a_3^{(r)}(n)\to \infty$ as $n \to \infty$. Also, we have $a_3^{(1)}(n)-b_3^{(1)}(n)\to \infty$ as $n\to \infty$.
    \item Furthermore, we have $$\displaystyle \lim_{r\to \infty}\left(\lim_{n\to \infty}\frac{b_3^{(r)}(n)}{a_3^{(r)}(n)}\right)=1.$$
\end{enumerate}
\end{theorem}
We next compare the total numbers of hooks of lengths $1$, $2$, and $3$ within the $A_r$-type partition family.
\begin{theorem}\label{p1_thm1.6} 
Let $r\geq 1$ be an integer.
\begin{enumerate}[label=(\roman*)]
    \item For sufficiently large $n$, we have $a_2^{(r)}(n)>a_1^{(r)}(n)$. Furthermore, we have $a_2^{(r)}(n)-a_1^{(r)}(n)\to\infty$ as $n\to\infty$.
    \item For sufficiently large $n$, we have $a_2^{(r)}(n)>a_3^{(r)}(n)$. Furthermore, we have $a_2^{(r)}(n)-a_3^{(r)}(n)\to\infty$ as $n\to\infty$.
    \item For sufficiently large $n$, we have $a_3^{(r)}(n)>a_1^{(r)}(n)$. Furthermore, we have $a_3^{(r)}(n)-a_1^{(r)}(n)\to\infty$ as $n\to\infty$.
\end{enumerate}
\end{theorem}
We obtain analogous comparisons within the $B_r$-type partition family. In particular, hooks of length $2$ eventually outnumber both hooks of lengths $1$ and $3$.
\begin{theorem}\label{p1_thm1.7}
Let $r\geq 1$ be an integer.
\begin{enumerate}[label=(\roman*)]
    \item For sufficiently large $n$, we have $b_2^{(r)}(n)>b_{\ell}^{(r)}(n)$, where $\ell=1,3$. Furthermore, for $\ell=1,3$, we have $b_2^{(r)}(n)-b_{\ell}^{(r)}(n) \to\infty$ as $n\to\infty$.
    \item For sufficiently large $n$, we have $b_3^{(r)}(n)>b_1^{(r)}(n)$. Furthermore, $b_3^{(r)}(n)-b_1^{(r)}(n)\to\infty$ as $n\to\infty$.
\end{enumerate}
\end{theorem}
Thus far, we have investigated hook length biases for fixed $r$ as $\ell$ varies. It is also natural to study these biases for fixed $\ell$ as $r$ varies. Let $\mathcal{A}_r(n)$ denote the set of partitions enumerated by $A_r(n)$. Since $\mathcal{A}_{r+1}(n) \subseteq \mathcal{A}_r(n)$ for all $n \geq 1$, we have $a_{\ell}^{(r)}(n) \ge a_{\ell}^{(r+1)}(n)$ for all $n,r\geq 1$. 
No analogous containment holds for the $B_r$-type partition families. This motivates us to investigate how their hook numbers vary with $r$. The following theorem establishes a strict asymptotic bias for each $\ell\in\{1,2,3\}$.

\begin{theorem}\label{p1_newthm} 
Let $r \geq 1$ be an integer. For each $\ell\in\{1,2,3\}$, there exists $N_\ell^{(r)}\in\mathbb{N}$ such that $b_\ell^{(r)}(n) > b_\ell^{(r+1)}(n)$
for all $n\geq N_\ell^{(r)}$. In particular, for each $\ell\in\{1,2,3\}$, as $n\to \infty$,
$$\frac{b_\ell^{(r)}(n)}{b_\ell^{(r+1)}(n)}\to \infty.$$
 Furthermore, $b_\ell^{(r)}(n) - b_\ell^{(r+1)}(n) \to \infty$ as $n \to \infty$.  
\end{theorem}

%%%%%%%%%%%%%%%%%%%%%%%%%%%%
\subsection{Asymptotics for $a_{\ell}^{(r)}(n)$ and $b_{\ell}^{(r)}(n)$}\label{Subsection-Asymp}

We prove Theorems~\ref{p1_thm1.3}--\ref{p1_newthm} by deriving asymptotic formulas for $a_{\ell}^{(r)}(n)$ and $b_{\ell}^{(r)}(n)$ for $\ell\in\{1,2,3\}$. Our approach relies on Wright's circle method. Throughout, we write $f(n)\sim g(n)$ as $n\to\infty$ if
$$\lim_{n\to\infty}\frac{f(n)}{g(n)}=1.$$
We also recall the \emph{digamma} function $\psi$, defined by
$$\psi(z) := \frac{\Gamma'(z)}{\Gamma(z)},$$ where $\Gamma$ is defined for $\text{Re}(z)>0$ by
$$\Gamma(z):=\int_{0}^{\infty}e^{-u}u^{z-1}du.$$ 
To express our asymptotic formulas, we introduce the following quantities involving the digamma function. For each positive integer $r$, let 
\begin{align*}&D_0^{(r)} := \psi\left(\frac{2}{2r+1}\right) - \psi\left(\frac{1}{2r+1}\right), \quad D_1^{(r)} := \psi\left(\frac{1}{2}+ \frac{1}{2r+1}\right) - \psi\left(\frac{1}{2r+1}\right), \\
&D_2^{(r)} := \psi\left(\frac{1}{2}+ \frac{2}{2r+1}\right) - \psi\left(\frac{2}{2r+1}\right), \quad \text{and} \quad D_3^{(r)} := \psi\left(\frac{1}{2}+ \frac{3}{2r+1}\right) - \psi\left(\frac{3}{2r+1}\right). 
\end{align*}
\par We begin by deriving asymptotic formulas for the total numbers of hooks of length $1$.
\begin{theorem}\label{p1_thm8}
    Let $r\geq 1$ be an integer. As $n \to \infty$, we have
    \begin{align*}
    a_{1}^{(r)}(n) \sim \frac{1}{\sqrt{2n\pi}}\ e^{\sqrt{\frac{4(r+1)\pi^2n}{6(2r+1)}}} \left( \frac{D_0^{(r)}}{2(2r+1)\sqrt{\pi}}   \right)
\end{align*}
and
\begin{align*}
b_{1}^{(r)}(n) \sim  \frac{1}{\sqrt{2n\pi}}\ e^{\sqrt{\frac{4(r+1)\pi^2n}{6(2r+1)}}} \left( \frac{r+1}{2\sqrt{\pi}(2r+1)}    \right).
\end{align*}
\end{theorem}
We next obtain asymptotic formulas for $a_{2}^{(r)}(n)$ and $b_{2}^{(r)}(n)$.
\begin{theorem}\label{p1_thm9}
Let $r\geq 1$ be an integer. As $n \to \infty$, we have
\begin{align*}
a_{2}^{(r)}(n) \sim  \frac{1}{\sqrt{2n\pi}}\ e^{\sqrt{\frac{4(r+1)\pi^2n}{6(2r+1)}}} \left( \frac{ 2(2r+1)(D_0^{(r)}+1) + (2r-3)(D_1^{(r)}-D_2^{(r)})}{4(2r+1)^2\sqrt{\pi}}  \right)
\end{align*}
and
\begin{align*}
b_{2}^{(r)}(n) \sim  \frac{1}{\sqrt{2n\pi}}\ e^{\sqrt{\frac{4(r+1)\pi^2n}{6(2r+1)}}} \left( \frac{3r+2}{4\sqrt{\pi}(2r+1)} \right). 
\end{align*}
\end{theorem}
Finally, we derive the corresponding asymptotic formulas for $a_{3}^{(r)}(n)$ and $b_{3}^{(r)}(n)$.
\begin{theorem}\label{p1_thm10}
Let $r\geq 1$ be an integer. As $n \to \infty$, we have
\begin{align*}
 a_{3}^{(r)}(n) &\sim \frac{1}{\sqrt{2n\pi}}\ e^{\sqrt{\frac{4(r+1)\pi^2n}{6(2r+1)}}}\left( \frac{(3+2\ln 2)(2r+1)^2+(8r^2-10r-1)D_1^{(r)}}{4(2r+1)^3\sqrt{\pi}} \right.\\
&\hspace{1 cm}\left.+ \frac{(-8r^2+28r-8)D_2^{(r)}+(4r^2-14r+10)D_3^{(r)}}{4(2r+1)^3\sqrt{\pi}}   \right)
\end{align*}
and
\begin{align*}
b_3^{(r)}(n) \sim \frac{1}{\sqrt{2n\pi}}\ e^{\sqrt{\frac{4(r+1)\pi^2n}{6(2r+1)}}}\left( \frac{16r+13}{24\sqrt{\pi}(2r+1)}   \right). 
\end{align*}
\end{theorem}
%%%%%%%%%%%%%%%%%%%%%%%
 \subsection{Generating functions for $a_{\ell}^{(r)}(n)$ and $b_{\ell}^{(r)}(n)$}
 We first recall the \emph{$q$-Pochhammer symbol} $(a;q)_n$ which is defined as follows:
 \[
(a;q)_n :=
\begin{cases}
1, & \text{if } n=0,\\[4pt]
(1-a)(1-aq)\cdots(1-aq^{n-1}), 
& \text{if } n\ge 1,\\[4pt]
\displaystyle\prod_{j=0}^{\infty}(1-aq^j), 
& \text{if } n=+\infty.
\end{cases}
\]
In this section, we find the generating functions of $a_{\ell}^{(r)}(n)$ and $b_{\ell}^{(r)}(n)$ for all $r\ge 1$ and $\ell=1, 2, 3$. These generating functions will be used to find biases and asymptotics for $a_{\ell}^{(r)}(n)$ and $b_{\ell}^{(r)}(n)$. In the following theorem, we find the generating function of $a_{\ell}^{(r)}(n)$ for $\ell=1$.
\begin{theorem}\label{p1_thm11}
 Let $r\geq 1$ be an integer. The generating function of $a_1^{(r)}(n)$ is given by 
\begin{align*}
\sum_{n\ge 0} a_1^{(r)}(n)q^n
&=
\frac{(-q^{2r+1};q^{2r+1})_\infty}{(q^2;q^2)_\infty} \sum_{m\ge0}   \left(\dfrac{q^{2(m+1)}(1+q^{(2r-1)(m+1)})}{1+q^{(2r+1)(m+1)}} \right).
\end{align*}
\end{theorem}
In the following theorem, we find the generating function of $a_{\ell}^{(r)}(n)$ for $\ell=2$.
\begin{theorem} \label{p1_thm12}
Let $r\geq 1$ be an integer. The generating function of $a_2^{(r)}(n)$ is given by
\begin{align*}
    \sum_{n\ge 0} a_2^{(r)}(n)q^n
&=
\frac{(-q^{2r+1};q^{2r+1})_\infty}{(q^2;q^2)_\infty} 
\sum_{m\ge0} 
\left(
\dfrac{q^{2(m+2)}(1-q^{2(m+1)})(1+q^{(2r-1)(m+2)})}{(1+q^{(2r+1)(m+1)})(1+q^{(2r+1)(m+2)})} \right.\\
&\hspace{2.7 cm } \left.+ \dfrac{q^{2(m+1)}(1+q^{(2r-1)(m+1)})}{1+q^{(2r+1)(m+1)}} 
\right).
\end{align*}
\end{theorem}
In the following theorem, we find the generating function of $a_{\ell}^{(r)}(n)$ for $\ell=3$.
\begin{theorem}\label{p1_thm13}
Let $r\geq 1$ be an integer. The generating function of $a_3^{(r)}(n)$ is given by
\begin{align*}
  \sum_{n\ge 0} a_3^{(r)}(n)q^n
&= \frac{(-q^{2r+1};q^{2r+1})_\infty}{(q^2;q^2)_\infty} \sum_{m\ge0} \left(
 \dfrac{q^{2(m+3)}(1-q^{2(m+1)})(1-q^{2(m+2)})(1+q^{(2r-1)(m+3)})}{(1+q^{(2r+1)(m+1)})(1+q^{(2r+1)(m+2)})(1+q^{(2r+1)(m+3)})}  \right.\\
&\hspace{1 cm}\left.+
\dfrac{q^{2(m+2)}(1-q^{2(m+1)})(1+q^{(2r-1)(m+2)})}{(1+q^{(2r+1)(m+1)})(1+q^{(2r+1)(m+2)})}
+ \dfrac{(q^{4(m+1)} +q^{(2r+1)(m+1)})}{(1+q^{(2r+1)(m+1)}) } \right).  
\end{align*} 
\end{theorem}
Next, we find the generating functions of $b_{\ell}^{(r)}(n)$ for all $r\ge 1$ and $\ell=1, 2, 3$. In the following theorem, we find the generating function of $b_{\ell}^{(r)}(n)$ for $\ell=1$.
\begin{theorem} \label{p1_thm14}
Let $r\geq 1$ be an integer. The generating function of $b_1^{(r)}(n)$ is given by
$$\sum_{n\ge 0} b_1^{(r)}(n)q^n
=  \frac{(-q^{2r+1};q^{2r+1})_\infty}{(q^2;q^2)_\infty}\left(
\frac{q^2 \sum_{k=0}^{2r}q^{2k} +q^{2r+1}}{1-q^{4r+2}} \right).$$
\end{theorem}
In the following theorem, we find the generating function of $b_{\ell}^{(r)}(n)$ for $\ell=2$.
\begin{theorem}\label{p1_thm15}
Let $r\geq 1$ be an integer. The generating function of $b_2^{(r)}(n)$ is given by
\begin{align*}
    \sum_{n\ge 0} b_2^{(r)}(n)q^n
&=\dfrac{(-q^{2r+1};q^{2r+1})_\infty}{(q^2;q^2)_\infty}
\left(
\dfrac{\sum_{k=1}^{2r+1}q^{2k} + q^{2r+1}}{1-q^{4r+2}}+
 \dfrac{q^{4r+2} + \sum_{k=1}^{2r+1}q^{4k}- q^{4r+1}-q^{4r+3}}{1-q^{8r+4}} 
\right).
\end{align*}
\end{theorem}
In the following theorem, we find the generating function of $b_{\ell}^{(r)}(n)$ for $\ell=3$.
\begin{theorem}\label{p1_thm16}
Let $r\geq 1$ be an integer. The generating function of $b_3^{(r)}(n)$ is given by
\begin{align*}
  \sum_{n\ge 0} b_3^{(r)}(n)q^n = &\frac{(-q^{2r+1};q^{2r+1})_\infty}{(q^2;q^2)_\infty}
\left(
\frac{q^{2r+1} + \sum_{k=0}^{2r}q^{2k+4}}{1 - q^{4r+2}} +  \frac{q^{4r+2} + \sum_{j=1}^{2r+1}q^{4j} - \sum_{k=0}^{2r}q^{4k+6}}{1 - q^{8r+4}} \right.\\
&\hspace{1.2cm}\left.+  \frac{2q^{6r+3} - (q^{6r+1} + q^{6r+2} + q^{6r+4} + q^{6r+5}) + \sum_{j=1}^{2r+1}q^{6j}}{1 - q^{12r+6}}
\right).
\end{align*}
\end{theorem}
The remainder of the paper is organized as follows. In Section~\ref{p1_sec2}, we derive the generating functions for $a_{\ell}^{(r)}(n)$ and $b_{\ell}^{(r)}(n)$, which denote the total number of hooks of length $\ell$ in all partitions of $n$ of $A_r$-type and $B_r$-type, respectively, for $\ell=1, 2, 3$. In Section~\ref{p1_sec3}, we establish asymptotic formulas for these functions for $\ell=1, 2, 3$. Finally, we prove the hook length bias results stated in Theorems~\ref{p1_thm1.3}--\ref{p1_newthm} in Section~\ref{p1_sec4}.

%%%%%%%%%%%%%%%%%%%%%%%%%%%%%%%%%%%%%%%%%%%%%%%%%%%%%%%%%%%%%%%%
\section{Proof of generating functions for $a_{\ell}^{(r)}(n)$ and $b_{\ell}^{(r)}(n)$}\label{p1_sec2}
In this section, we derive the generating functions for $a_{\ell}^{(r)}(n)$ and $b_{\ell}^{(r)}(n)$ for all integers $r\geq 1$ and $\ell=1,2,3$. Recall that, for a box $v$ in the Young diagram of a partition $\lambda$, $a_\lambda(v)$ denotes the \emph{arm length} of $v$ and $l_\lambda(v)$ denotes the \emph{leg length} of $v$.
\par
The generating function for $A_r(n)$ is $\displaystyle\frac{(-q^{2r+1};q^{2r+1})_\infty}{(q^2;q^2)_\infty}$, while for $B_r(n)$ it is equal to $\displaystyle\frac{1}{(q^2;q^2)_\infty\,(q^{2r+1};q^{4r+2})_\infty}$. Since $A_r(n)=B_r(n)$ for all $n\ge0$, these generating functions are equal and hence
$$\frac{(-q^{2r+1};q^{2r+1})_\infty}{(q^2;q^2)_\infty} = \frac{1}{(q^2;q^2)_\infty\,(q^{2r+1};q^{4r+2})_\infty}.$$
This $q$-series identity will be used in finding the generating functions for hook numbers below. 

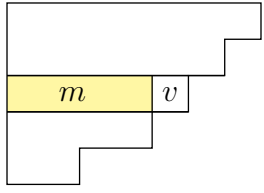
\begin{figure}[htbp]
\centering

%========================================================
% CASE 1
%========================================================

\begin{tikzpicture}[
    x=0.48cm,
    y=0.48cm,
    line width=0.5pt
]

%--------------------------------------------------------
% Shaded hook -- fill first
%--------------------------------------------------------
\fill[yellow!45] (0,2) rectangle (4,3);

%--------------------------------------------------------
% Young diagram boundary
%--------------------------------------------------------
\draw
(0,0) -- (2,0)
-- (2,1)
-- (3,1)
-- (4,1)
-- (4,2)
-- (5,2)
-- (5,3)
-- (6,3)
-- (6,4)
-- (7,4)
-- (7,5)
-- (0,5)
-- cycle;

%--------------------------------------------------------
% Horizontal row boundaries
%--------------------------------------------------------
\draw (0,2) -- (4,2);
\draw (0,3) -- (6,3);

%--------------------------------------------------------
% Boundaries of shaded hook
%--------------------------------------------------------
\draw (0,2) -- (5,2);
\draw (0,3) -- (5,3);

% Vertical boundaries

\draw (4,2) -- (4,3);

%--------------------------------------------------------
% Labels
%--------------------------------------------------------
\node at (1.8,2.5) {$m$};
\node at (4.5,2.5) {$v$};

\end{tikzpicture}
\hspace{1.0cm}

%========================================================
% CASE LABELS
%========================================================

\caption{Case for a hook of length $1$, $a_\lambda(v) = 0$ and $l_\lambda(v) = 0$.}
\label{fig:hook-length-1}

\end{figure}

\begin{proof}[Proof of Theorem \ref{p1_thm11}]
     To count the hooks of length $1$ in a partition $\lambda$, it suffices to count the number of distinct part sizes in $\lambda$. Consequently, if the hook length of $v$ is 1, then $a_\lambda(v) = l_\lambda(v) = 0$ (see Fig.~\ref{fig:hook-length-1}). Moreover, for each $m\ge 0$, the part $m+1$ must appear at least twice in $\lambda$; if it appears an odd number of times, then it appears at least $2r+1$ times. Therefore, the generating function takes the form
\begin{align*}
&\sum_{n\ge 0} a_1^{(r)}(n)q^n \\
&=\sum_{m\ge0} \left( q^{2(m+1)}+q^{4(m+1)}+\cdots + q^{2r(m+1)} + q^{(2r+1)(m+1)}+  q^{(2r+2)(m+1)} + \cdots \right)\\
&\hspace{.5cm}\times \prod_{\substack{i\ge1\\ i\ne m+1}}
\left( 1 + q^{2i}+q^{4i}+\cdots +q^{2ri} + q^{(2r+1)i}+  q^{(2r+2)i} + \cdots \right)\\
&=\sum_{m\ge0} \left(\dfrac{q^{2(m+1)}(1+ q^{(2r-1)(m+1)})}{1-q^{2(m+1)}}\right)
\prod_{\substack{i\ge1\\ i\ne m+1}}
\left(  \dfrac{1 + q^{(2r+1)i}}{1-q^{2i}}\right)\\
&=\sum_{m\ge0} \left(\dfrac{q^{2(m+1)}(1+ q^{(2r-1)(m+1)})}{1-q^{2(m+1)}}\right)
\left( \frac{(-q^{2r+1};q^{2r+1})_\infty}{(q^2;q^2)_\infty}\left(\dfrac{1-q^{2(m+1)}}{1+q^{(2r+1)(m+1)}} \right)\right)\\
&=\frac{(-q^{2r+1};q^{2r+1})_\infty}{(q^2;q^2)_\infty} \sum_{m\ge0}   \dfrac{q^{2(m+1)}(1+q^{(2r-1)(m+1)})}{1+q^{(2r+1)(m+1)}}
\end{align*}
which is the required generating function for $a_1^{(r)}(n)$.
\end{proof}  

\begin{figure}[htbp]
\centering

%=========================================================
% CASE (1)
%=========================================================
\begin{tikzpicture}[
    x=0.55cm,
    y=0.55cm,
    line width=0.5pt
]

%---------------------------------------------------------
% Gray shaded cells -- DRAW FIRST
%---------------------------------------------------------
\fill[yellow!45] (0,2) rectangle (3,3);
\fill[yellow!45] (4,2) rectangle (5,3);

%---------------------------------------------------------
% Complete Young diagram boundary -- DRAW ON TOP
%---------------------------------------------------------
\draw
(0,0) -- (1,0)
-- (2,0)
-- (2,1)
-- (3,1)
-- (3,2)
-- (4,2)
-- (5,2)
-- (5,3)
-- (6,3)
-- (6,4)
-- (7,4)
-- (7,5)
-- (0,5)
-- cycle;

%---------------------------------------------------------
% Horizontal boundaries of the Young diagram
%---------------------------------------------------------
\draw (0,2) -- (5,2);
\draw (0,3) -- (5,3);

%---------------------------------------------------------
% Boundaries of the shaded cells
%---------------------------------------------------------
\draw (3,2) -- (3,3);
\draw (4,2) -- (4,3);

%---------------------------------------------------------
% Labels
%---------------------------------------------------------
\node at (1.5,2.5) {$m$};
\node at (3.5,2.5) {$v$};
\node at (4.5,2.5) {$1$};

\end{tikzpicture}
\hspace{1.7cm}
%
%=========================================================
% CASE (2)
%=========================================================
\begin{tikzpicture}[
    x=0.55cm,
    y=0.55cm,
    line width=0.5pt
]

%---------------------------------------------------------
% Gray shaded cells -- DRAW FIRST
%---------------------------------------------------------
\fill[yellow!45] (0,2) rectangle (3,3);
\fill[yellow!45] (3,1) rectangle (4,2);

%---------------------------------------------------------
% Complete Young diagram boundary -- DRAW ON TOP
%---------------------------------------------------------
\draw
(0,0) -- (2,0)
-- (2,1)
-- (3,1)
-- (4,1)
-- (4,2)
-- (4,3)
-- (6,3)
-- (6,4)
-- (7,4)
-- (7,5)
-- (0,5)
-- cycle;

%---------------------------------------------------------
% Horizontal boundaries
%---------------------------------------------------------
\draw (0,2) -- (4,2);
\draw (0,3) -- (6,3);

%---------------------------------------------------------
% Boundaries of shaded cells
%---------------------------------------------------------
\draw (3,2) -- (3,3);
\draw (3,1) -- (3,2);

%---------------------------------------------------------
% Labels
%---------------------------------------------------------
\node at (1.5,2.5) {$m$};
\node at (3.5,2.5) {$v$};
\node at (3.5,1.5) {$1$};

\end{tikzpicture}

\vspace{0.25cm}

%=========================================================
% LABELS BELOW THE FIGURES
%=========================================================

\begin{minipage}{0.42\textwidth}
\centering
\[
(1)\quad
a_{\lambda}(v)=1,\quad
l_{\lambda}(v)=0
\]
\end{minipage}
\hspace{0.05cm}
\begin{minipage}{0.42\textwidth}
\centering
\[
(2)\quad
a_{\lambda}(v)=0,\quad
l_{\lambda}(v)=1
\]
\end{minipage}

\caption{Cases for a hook of length $2$.}
\label{fig:hook-length-2}
\end{figure}
    \begin{proof}[Proof of Theorem \ref{p1_thm12}]
    To count the hooks of length 2, we have two cases (see Fig.~\ref{fig:hook-length-2}):
    \begin{itemize}
        \item Case 1: partitions include a box $v$ such that $a_\lambda(v) = 1$ and $l_\lambda(v) = 0$,
        \item Case 2: partitions include a box $v$ such that $a_\lambda(v) = 0$ and $l_\lambda(v) = 1$.
    \end{itemize}

    \textbf{Case 1:} For each $m\geq 0$, a partition in this case has a part of length $m+2$ that is repeated at least twice but, when it appears odd many times then the number of appearances is at least $2r+1$, and contains no part of length $m+1$.
 Thus, the  generating function is 
\begin{align*}
&\sum_{m\ge0} \left( q^{2(m+2)}+q^{4(m+2)}+\cdots +q^{2r(m+2)} + q^{(2r+1)(m+2)}+  q^{(2r+2)(m+2)} + \cdots \right) \\
&\times \prod_{\substack{i\ge1\\ i\ne m+1,m+2}}
\left( 1 + q^{2i}+q^{4i}+\cdots +q^{2ri} + q^{(2r+1)i}+  q^{(2r+2)i} + \cdots \right)\\
&=\sum_{m\ge0} \left(\dfrac{q^{2(m+2)}(1+ q^{(2r-1)(m+2)})}{1-q^{2(m+2)}}\right)
\prod_{\substack{i\ge1\\ i\ne m+1,m+2}}
\left(  \dfrac{1 + q^{(2r+1)i}}{1-q^{2i}}\right)\\
&=\sum_{m\ge0} \left(\dfrac{q^{2(m+2)}(1+ q^{(2r-1)(m+2)})}{1-q^{2(m+2)}}\right)
\left( \frac{(-q^{2r+1};q^{2r+1})_\infty}{(q^2;q^2)_\infty} \dfrac{(1-q^{2(m+1)})(1-q^{2(m+2)})}{(1+q^{(2r+1)(m+1)})(1+q^{(2r+1)(m+2)})}  \right)\\
&=\frac{(-q^{2r+1};q^{2r+1})_\infty}{(q^2;q^2)_\infty} \sum_{m\ge0}   \dfrac{q^{2(m+2)}(1-q^{2(m+1)})(1+q^{(2r-1)(m+2)})}{(1+q^{(2r+1)(m+1)})(1+q^{(2r+1)(m+2)})}.
\end{align*}
 \textbf{Case 2:} For each $m\geq 0$, a partition in this case contains a part of size $m+1$ that appears at least twice but, when it appears odd many times then the number of appearance is at least $2r+1$. Accordingly, the generating function for this case is identical to the generating function for $a_1^{(r)}(n)$, which is given by
\begin{align*}
    \frac{(-q^{2r+1};q^{2r+1})_\infty}{(q^2;q^2)_\infty} \sum_{m\ge0}   \dfrac{q^{2(m+1)}(1+q^{(2r-1)(m+1)})}{1+q^{(2r+1)(m+1)}}. 
\end{align*}
Adding these two generating functions yields the required generating function for $a_2^{(r)}(n)$.
\end{proof}

\begin{proof}[Proof of Theorem \ref{p1_thm13}]
    To count hooks of length 3, we consider three cases as in Fig.~\ref{fig:hook-length-3}:
\begin{itemize}
\item Case 1: partitions include a box with $a_\lambda(v) = 2$ and $l_\lambda(v) = 0$,
\item Case 2: partitions include a box with $a_\lambda(v) = 1$ and $l_\lambda(v) = 1$,
\item Case 3: partitions include a box with $a_\lambda(v) = 0$ and $l_\lambda(v) = 2$.
\end{itemize}

\begin{figure}[htbp]
\centering

%========================================================
% CASE 1
%========================================================
\begin{tikzpicture}[
    x=0.48cm,
    y=0.48cm,
    line width=0.5pt
]

%--------------------------------------------------------
% Shaded hook -- fill first
%--------------------------------------------------------
\fill[yellow!45] (0,2) rectangle (3,3);
\fill[yellow!45] (4,2) rectangle (6,3);
%--------------------------------------------------------
% Young diagram boundary
% Shape: (7,6,5,3,2,1)
%--------------------------------------------------------
\draw
(0,0) -- (1,0)
-- (2,0)
-- (2,1)
-- (3,1)
-- (3,2)
-- (6,2)
-- (6,3)
-- (6,4)
-- (7,4)
-- (7,5)
-- (8,5)
-- (8,6)
-- (0,6)
-- cycle;

%--------------------------------------------------------
% Horizontal row boundaries
%--------------------------------------------------------
\draw (0,3) -- (6,3);

%--------------------------------------------------------
% Boundaries of shaded hook
%--------------------------------------------------------
\draw (0,2) -- (6,2);
\draw (0,3) -- (6,3);

% Cell boundaries inside shaded hook
\draw (3,2) -- (3,3);
\draw (4,2) -- (4,3);

%--------------------------------------------------------
% Labels
%--------------------------------------------------------
\node at (1.5,2.5) {$m$};
\node at (3.5,2.5) {$v$};
\node at (5,2.5) {$2$};

\end{tikzpicture}
\hspace{1.0cm}
%
%========================================================
% CASE 2
%========================================================
\begin{tikzpicture}[
    x=0.48cm,
    y=0.48cm,
    line width=0.5pt
]

%--------------------------------------------------------
% Shaded hook -- fill first
%--------------------------------------------------------
\fill[yellow!45] (0,2) rectangle (3,3);
\fill[yellow!45] (4,2) rectangle (5,3);
\fill[yellow!45] (3,1) rectangle (4,2);

%--------------------------------------------------------
% Young diagram boundary
%--------------------------------------------------------
\draw
(0,0) -- (2,0)
-- (2,1)
-- (3,1)
-- (4,1)
-- (4,2)
-- (5,2)
-- (5,3)
-- (6,3)
-- (6,4)
-- (7,4)
-- (7,5)
-- (0,5)
-- cycle;

%--------------------------------------------------------
% Horizontal row boundaries
%--------------------------------------------------------
\draw (0,2) -- (4,2);
\draw (0,3) -- (6,3);

%--------------------------------------------------------
% Boundaries of shaded hook
%--------------------------------------------------------
\draw (0,2) -- (5,2);
\draw (0,3) -- (5,3);

% Vertical boundaries
\draw (3,2) -- (3,3);
\draw (4,2) -- (4,3);

% Boundary of lower shaded cell
\draw (3,1) -- (4,1);
\draw (3,1) -- (3,2);
\draw (4,1) -- (4,2);

%--------------------------------------------------------
% Labels
%--------------------------------------------------------
\node at (1.5,2.5) {$m$};
\node at (3.5,2.5) {$v$};
\node at (4.5,2.5) {$1$};

\node at (3.5,1.5) {$1$};
\node at (4.5,1.5) {$*$};

\end{tikzpicture}
\hspace{1.0cm}
%
%========================================================
% CASE 3
%========================================================
\begin{tikzpicture}[
    x=0.48cm,
    y=0.48cm,
    line width=0.5pt
]

%--------------------------------------------------------
% Shaded vertical hook -- fill first
%--------------------------------------------------------
\fill[yellow!45] (0,3) rectangle (3,4);
\fill[yellow!45] (3,1) rectangle (4,3);

%--------------------------------------------------------
% Young diagram boundary
%--------------------------------------------------------
\draw
(0,0) -- (2,0)
-- (2,1)
-- (3,1)
-- (4,1)
-- (4,2)
-- (4,3)
-- (4,4)
-- (5,4)
-- (5,5)
-- (6,5)
-- (6,6)
-- (0,6)
-- cycle;

%--------------------------------------------------------
% Horizontal row boundaries
%--------------------------------------------------------
\draw (0,3) -- (4,3);
\draw (0,4) -- (5,4);

%--------------------------------------------------------
% Boundaries of shaded hook
%--------------------------------------------------------
\draw (0,3) -- (3,3);
\draw (0,4) -- (3,4);

% Vertical boundaries of hook
\draw (3,1) -- (3,4);
\draw (4,1) -- (4,4);

% Internal horizontal boundaries
\draw (3,3) -- (4,3);

%--------------------------------------------------------
% Labels
%--------------------------------------------------------
\node at (1.5,3.5) {$m$};
\node at (3.5,3.5) {$v$};
\node at (3.5,2) {$2$};

\end{tikzpicture}

\vspace{0.25cm}

%========================================================
% CASE LABELS
%========================================================

\begin{minipage}{0.30\textwidth}
\centering
\[
(1)\quad
a_{\lambda}(v)=2,\
l_{\lambda}(v)=0
\]
\end{minipage}
\hfill
\begin{minipage}{0.30\textwidth}
\centering
\[
(2)\quad
a_{\lambda}(v)=1, \
l_{\lambda}(v)=1
\]
\end{minipage}
\hfill
\begin{minipage}{0.30\textwidth}
\centering
\[
(3)\quad
a_{\lambda}(v)=0,\
l_{\lambda}(v)=2
\]
\end{minipage}

\caption{Cases for a hook of length $3$.}
\label{fig:hook-length-3}
\end{figure}

\textbf{Case 1:} Each partition within this case contains a part of length $m+3$ that is repeated at least twice, but, when it appears odd many times, the number of appearances is at least $2r+1$ and has no parts of length $m+1$ or $m+2$, for each $m \ge 0$. Thus, the generating function corresponding to this case is given by
\begin{align*}
&\sum_{m\ge0} \left( q^{2(m+3)}+q^{4(m+3)}+\cdots +q^{2r(m+3)} + q^{(2r+1)(m+3)}+  q^{(2r+2)(m+3)} + \cdots \right) \\
&\times \prod_{\substack{i\ge1\\ i\ne m+1,m+2,m+3}}
\left( 1 + q^{2i}+q^{4i}+\cdots +q^{2ri} + q^{(2r+1)i}+  q^{(2r+2)i} + \cdots \right)\\
&=\sum_{m\ge0} \left(\dfrac{q^{2(m+3)}(1+ q^{(2r-1)(m+3)})}{1-q^{2(m+3)}}\right)
\prod_{\substack{i\ge1\\ i\ne m+1,m+2,m+3}}
\left(  \dfrac{1 + q^{(2r+1)i}}{1-q^{2i}}\right)\\
&=\sum_{m\ge0} \left(\dfrac{q^{2(m+3)}(1+ q^{(2r-1)(m+3)})}{1-q^{2(m+3)}}\right)
\left( \frac{(-q^{2r+1};q^{2r+1})_\infty}{(q^2;q^2)_\infty}\left(\dfrac{1-q^{2(m+1)}}{1+q^{(2r+1)(m+1)}} \right) \right.\\
&\hspace{2 cm}\left. \left(\dfrac{1-q^{2(m+2)}}{1+q^{(2r+1)(m+2)}} \right)  \left(\dfrac{1-q^{2(m+3)}}{1+q^{(2r+1)(m+3)}} \right) \right)\\
&=\frac{(-q^{2r+1};q^{2r+1})_\infty}{(q^2;q^2)_\infty} \sum_{m\ge0}\dfrac{q^{2(m+3)}(1-q^{2(m+1)})(1-q^{2(m+2)})(1+q^{(2r-1)(m+3)})}{(1+q^{(2r+1)(m+1)})(1+q^{(2r+1)(m+2)})(1+q^{(2r+1)(m+3)})}.
\end{align*}

 \textbf{Case 2:} In this case, the partitions contain a part of length $m+2$. Importantly, the $*$ position must always be filled; leaving it empty would indicate that the part $m+1$ appears only once, which is not allowed for any partition of type $A_r$. Consequently, the generating function for this case is identical to that of Case 1 for $a_2^{(r)}(n)$ and is given by

 \begin{align*}
  \frac{(-q^{2r+1};q^{2r+1})_\infty}{(q^2;q^2)_\infty} \sum_{m\ge0}   \dfrac{q^{2(m+2)}(1-q^{2(m+1)})(1+q^{(2r-1)(m+2)})}{(1+q^{(2r+1)(m+1)})(1+q^{(2r+1)(m+2)})}. 
 \end{align*}

\textbf{Case 3:} Each  partition in this case contains a part of length $m+1$ that is repeated at least three times, but, when it appears odd many times, the number of appearances is at least $2r+1$ for each $m \ge 0$. Thus, the generating function corresponding to this case is given by
\begin{align*}
&\sum_{m\ge0} \left( \left( q^{4(m+1)}+q^{6(m+1)}+ q^{8(m+1)} +\cdots \right) + \left( q^{(2r+1)(m+1)}+  q^{(2r+3)(m+1)} + \cdots \right) \right)\\
&\times \prod_{\substack{i\ge1\\ i\ne m+1}}
\left( 1 + q^{2i}+q^{4i}+\cdots + q^{2ri} + q^{(2r+1)i}+  q^{(2r+2)i} + \cdots \right)\\
&=\sum_{m\ge0} \left(\dfrac{q^{4(m+1)} + q^{(2r+1)(m+1)} }{1-q^{2(m+1)}}\right)
\prod_{\substack{i\ge1\\ i\ne m+1}}
\left(  \dfrac{1 + q^{(2r+1)i}}{1-q^{2i}}\right)\\
&=\sum_{m\ge0}\left(\dfrac{q^{4(m+1)} + q^{(2r+1)(m+1)} }{1-q^{2(m+1)}}\right)
\left( \frac{(-q^{2r+1};q^{2r+1})_\infty}{(q^2;q^2)_\infty}\left(\dfrac{1-q^{2(m+1)}}{1+q^{(2r+1)(m+1)}} \right) \right)\\
&=\frac{(-q^{2r+1};q^{2r+1})_\infty}{(q^2;q^2)_\infty} \sum_{m\ge0}   \dfrac{q^{4(m+1)} +q^{(2r+1)(m+1)}}{1+q^{(2r+1)(m+1)}}.
\end{align*}
Adding these three generating functions produces the generating function for $a_3^{(r)}(n)$.
\end{proof}
Next, we derive the generating functions for $b_1^{(r)}(n)$, $b_2^{(r)}(n)$, and $b_3^{(r)}(n)$ for all positive integers $r$. Set $S := \{ n \in \mathbb{N} : n \text{~is even or~} n \equiv 2r+1 \pmod{4r+2}\}$.
\begin{proof}[Proof of Theorem~\ref{p1_thm14}]
Following similarly as in the proof of Theorem \ref{p1_thm11}, and using the fact that $a_\lambda(v) = l_\lambda(v) = 0$ whenever the hook length of $v$ is $1$, we obtain the following generating function:
\begin{align*}
\sum_{n\ge 0} b_1^{(r)}(n)q^n 
&=\sum_{\substack{m\ge0\\ m+1 \in S } }  \left(q^{(m+1)}+q^{2(m+1)}+\cdots\right)\prod_{\substack{ i \in S\\ i\ne m+1 }}
\left(1+q^{i}+q^{2i}+\cdots\right)\\
&=\sum_{\substack{m\ge0\\ m+1 \in S } } \left(\dfrac{q^{m+1}}{1-q^{m+1}}\right)
\prod_{\substack{ i \in S\\ i\ne m+1}}
\left( \dfrac{1}{1-q^{i}}\right)\\
&=\sum_{\substack{m\ge0\\ m+1 \in S } } \left(\dfrac{q^{m+1}}{1-q^{m+1}}\right)
\left(  \dfrac{(-q^{2r+1};q^{2r+1})_\infty}{(q^2;q^2)_\infty}(1-q^{m+1})   \right)\\
&=\frac{(-q^{2r+1};q^{2r+1})_\infty}{(q^2;q^2)_\infty} \sum_{\substack{m\ge0\\ m+1 \in S } } q^{m+1}\\
&= \frac{(-q^{2r+1};q^{2r+1})_\infty}{(q^2;q^2)_\infty}\left(
\frac{q^2 \sum_{k=0}^{2r}q^{2k} +q^{2r+1}}{1-q^{4r+2}}
\right).
\end{align*}
This completes the proof of the theorem.
\end{proof}
\begin{proof}[Proof of Theorem~\ref{p1_thm15}] As in Theorem \ref{p1_thm12}, to count the number of $2$-hooks, we need to consider the following two cases:
     \begin{itemize}
     \item Case 1: partitions include a box such that $a_\lambda(v) = 1$ and $l_\lambda(v) = 0$,
    \item Case 2: partitions include a box such that $a_\lambda(v) = 0$ and $l_\lambda(v) = 1$.
    \end{itemize}
\textbf{Case 1:} Each partition in this case possesses a part of size $m+2$ occurring at least once, where $m+2 \in S$, and contains no part of size $m+1$, for all $m \ge 0$. Thus, in this case, the  generating function is 

\begin{align*}
&\sum_{\substack{m\ge0\\ m+2 \in S}}  \left(q^{(m+2)}+q^{2(m+2)}+\cdots\right)
\prod_{\substack{i\in S\\ i\ne m+1,m+2}}
\left(1+q^{i}+q^{2i}+\cdots\right)\\
&=\sum_{\substack{m\ge0\\ m+2 \in S}} \left(\dfrac{q^{m+2}}{1-q^{m+2}}\right)
\prod_{\substack{i\in S\\ i\ne m+1,m+2}}
\left( \dfrac{1}{1-q^{i}}\right)\\
&=\sum_{\substack{m\ge0\\ m+1,m+2\in S}}
\left(\dfrac{q^{m+2}}{1-q^{m+2}}\right)
\dfrac{(-q^{2r+1};q^{2r+1})_\infty}{(q^2;q^2)_\infty}(1-q^{m+1})(1-q^{m+2})\\
&\qquad+\sum_{\substack{m\ge0\\ m+2\in S\\m+1\notin S}}
\left(\dfrac{q^{m+2}}{1-q^{m+2}}\right)
\dfrac{(-q^{2r+1};q^{2r+1})_\infty}{(q^2;q^2)_\infty}(1-q^{m+2})
\\
&=\frac{(-q^{2r+1};q^{2r+1})_\infty}{(q^2;q^2)_\infty}
\left( \sum_{\substack{m\ge0\\ m+1,m+2\in S}}
q^{m+2}(1-q^{m+1})
+
\sum_{\substack{m\ge0\\ m+2\in S\\ m+1\notin S}}
q^{m+2} \right)\\
&=\frac{(-q^{2r+1};q^{2r+1})_\infty}{(q^2;q^2)_\infty}
\left(
\dfrac{\sum_{k=1}^{2r+1}q^{2k} + q^{2r+1}}{1-q^{4r+2}} - \dfrac{q^{4r+1}+ q^{4r+3}}{1-q^{8r+4}}
\right).
\end{align*}

\textbf{Case 2:} Partitions falling under this case possess a part of size $m+1$ occurring at least twice, where $m+1 \in S$, for all $m \ge 0$. Consequently, the generating function is given by
\begin{align*}
&\sum_{\substack{m\ge0\\ m+1 \in S}} \left(q^{2(m+1)}+q^{3(m+1)}+\cdots\right)
\prod_{\substack{i\in S\\ i\ne m+1}}
\left(1+q^{i}+q^{2i}+\cdots\right)\\
&=\sum_{\substack{m\ge0\\ m+1 \in S}} \left(\dfrac{q^{2(m+1)}}{1-q^{m+1}}\right)
\prod_{\substack{ i\in S\\ i\ne m+1}}
\left( \dfrac{1}{1-q^{i}}\right)\\
&=\sum_{\substack{m\ge0\\ m+1 \in S}} \left(\dfrac{q^{2(m+1)}}{1-q^{m+1}}\right)
\left( \dfrac{(-q^{2r+1};q^{2r+1})_\infty}{(q^2;q^2)_\infty}(1-q^{m+1})  \right)\\
&=\frac{(-q^{2r+1};q^{2r+1})_\infty}{(q^2;q^2)_\infty} \sum_{\substack{m\ge0\\ m+1 \in S}} q^{2(m+1)}\\
&=\frac{(-q^{2r+1};q^{2r+1})_\infty}{(q^2;q^2)_\infty} \left(  \dfrac{q^{4r+2}+ \sum_{k=1}^{2r+1}q^{4k}}{1-q^{8r+4}} \right). 
\end{align*}
We obtain the required generating function for $b_2^{(r)}(n)$ by adding these two generating functions.
\end{proof}
\begin{proof}[Proof of Theorem~\ref{p1_thm16}]
As shown in Fig.~\ref{fig:hook-length-3}, to count the number of hooks of length 3, we need to consider the following three cases.
\begin{itemize}
\item Case 1: partitions include a box with $a_\lambda(v) = 2$ and $l_\lambda(v) = 0$,
\item Case 2: partitions include a box with $a_\lambda(v) = 1$ and $l_\lambda(v) = 1$,
\item Case 3: partitions include a box with $a_\lambda(v) = 0$ and $l_\lambda(v) = 2$.
    \end{itemize}
\textbf{Case 1:} Each partition in this case contains a part of size $m+3$, where $m+3\in S$, and contains no part of size $m+1$ or $m+2$, for all $m\ge 0$. Thus, the generating function corresponding to this case is given by
 \begin{align*}
&\sum_{\substack{m\ge0\\ m+3 \in S}}  \left(q^{(m+3)}+q^{2(m+3)}+\cdots\right)
\prod_{\substack{i\in S\\ i\ne m+1,m+2,m+3}}
\left(1+q^{i}+q^{2i}+\cdots\right)\\
&=\sum_{\substack{m\ge0\\ m+3 \in S}} \left(\dfrac{q^{m+3}}{1-q^{m+3}}\right)
\prod_{\substack{i\in S\\ i\ne m+1,m+2,m+3}}
\left( \dfrac{1}{1-q^{i}}\right)\\
&=  \sum_{\substack{m\ge0\\m+1, m+2,m+3 \in S}} \left(\dfrac{q^{m+3}}{1-q^{m+3}}\right)
\dfrac{(-q^{2r+1};q^{2r+1})_\infty}{(q^2;q^2)_\infty}(1-q^{m+1})(1-q^{m+2})(1-q^{m+3}) \\
&\qquad + \sum_{\substack{m\ge0\\m+2, m+3 \in S\\ m+1 \notin S}} \left(\dfrac{q^{m+3}}{1-q^{m+3}}\right)
\dfrac{(-q^{2r+1};q^{2r+1})_\infty}{(q^2;q^2)_\infty}(1-q^{m+2})(1-q^{m+3}) \\
&\qquad + \sum_{\substack{m\ge0\\m+1, m+3 \in S\\ m+2 \notin S } } \left(\dfrac{q^{m+3}}{1-q^{m+3}}\right)
\dfrac{(-q^{2r+1};q^{2r+1})_\infty}{(q^2;q^2)_\infty}(1-q^{m+1})(1-q^{m+3}) \\
&=\frac{(-q^{2r+1};q^{2r+1})_\infty}{(q^2;q^2)_\infty}
\left( \sum_{\substack{m\ge0\\m+1, m+2,m+3 \in S}} q^{m+3}(1-q^{m+1})(1-q^{m+2}) \right.\\
&\hspace{1.2cm}\left.+ \sum_{\substack{m\ge0\\m+2, m+3 \in S\\ m+1 \notin S}}
q^{m+3}(1-q^{m+2}) + \sum_{\substack{m\ge0\\m+1, m+3 \in S\\ m+2\notin S } }
q^{m+3}(1-q^{m+1}) 
\right)\\
&=\dfrac{(-q^{2r+1};q^{2r+1})_\infty}{(q^2;q^2)_\infty}
\left( 
\frac{q^{2r+1}+ \sum_{k=0}^{2r}q^{2k+4} }{1-q^{4r+2}} - \dfrac{q^{4r+1}+q^{4r+3}+  \sum_{k=0}^{2r}q^{4k+6} }{1-q^{8r+4}} + \dfrac{q^{6r+3}}{1-q^{12r+6}}
\right). 
\end{align*}

\textbf{Case 2:} We consider two subcases depending on whether the $*$ position is occupied. If $*$ position is empty, then the partition contains a part of size $m+2$ at least once and a part of size $m+1$ exactly once, with $m+1, m+2 \in S$ for all $m \ge 0$. Therefore, the generating function in this situation is

\begin{align*}
&\sum_{\substack{m\ge0\\ m+1,m+2 \in S } }  q^{m+1}\left(q^{(m+2)}+q^{2(m+2)}+\cdots\right)
\prod_{\substack{i\in S\\ i\ne m+1,m+2}}
\left(1+q^{i}+q^{2i}+\cdots\right)\\
&=\sum_{\substack{m\ge0\\ m+1,m+2 \in S } } q^{m+1} \left(\dfrac{q^{m+2}}{1-q^{m+2}}\right)
\prod_{\substack{i\in S\\ i\ne m+1,m+2}}
\left( \dfrac{1}{1-q^{i}}\right)\\
&=\sum_{\substack{m\ge0\\ m+2,m+1 \in S } } \left(\dfrac{q^{2m+3}}{1-q^{m+2}}\right)
 \frac{(-q^{2r+1};q^{2r+1})_\infty}{(q^2;q^2)_\infty}(1-q^{m+1})(1-q^{m+2}) \\
&=\frac{(-q^{2r+1};q^{2r+1})_\infty}{(q^2;q^2)_\infty} \sum_{\substack{m\ge0\\ m+1,m+2 \in S } } q^{2m+3}(1-q^{m+1})\\
&=\dfrac{(-q^{2r+1};q^{2r+1})_\infty}{(q^2;q^2)_\infty}\left(  \dfrac{q^{4r+1} +q^{4r+3}}{1-q^{8r+4}} - \dfrac{q^{6r+1}+ q^{6r+4}}{1-q^{12r+6}}\right).
\end{align*}
However, if $*$ position is occupied, then the partition contains at least two parts of size $m+2$ and no parts of size $m+1$, where $m+2\in S$ for all $m \ge 0$. Accordingly, the generating function is

\begin{align*}
&\sum_{\substack{m\ge0\\ m+2 \in S } }  \left(q^{2(m+2)}+q^{3(m+2)}+\cdots\right)
\prod_{\substack{i\in S\\ i\ne m+1,m+2}}
\left(1+q^{i}+q^{2i}+\cdots\right)\\
&=\sum_{\substack{m\ge0\\ m+2 \in S } }  \left(\dfrac{q^{2(m+2)}}{1-q^{m+2}}\right)
\prod_{\substack{i\in S\\ i\ne m+1,m+2}}
\left( \dfrac{1}{1-q^{i}}\right)\\
&=\sum_{\substack{m\ge0\\ m+1,m+2 \in S } } \left(\dfrac{q^{2(m+2)}}{1-q^{m+2}}\right)
 \frac{(-q^{2r+1};q^{2r+1})_\infty}{(q^2;q^2)_\infty}(1-q^{m+1})(1-q^{m+2}) \\  &\qquad +  \sum_{\substack{m\ge0\\ m+2 \in S\\m+1 \notin S } } \left(\dfrac{q^{2(m+2)}}{1-q^{m+2}}\right)
 \dfrac{(-q^{2r+1};q^{2r+1})_\infty}{(q^2;q^2)_\infty}(1-q^{m+2}) \\
&=\frac{(-q^{2r+1};q^{2r+1})_\infty}{(q^2;q^2)_\infty} \left(  \dfrac{q^{4r+2} + \sum_{j= 1}^{2r+1}q^{4j}}{1-q^{8r+4}} - \dfrac{q^{6r+2}+ q^{6r+5}}{1-q^{12r+6}}\right).
\end{align*}

\textbf{Case 3:} In this case, each partition contains a part of size $m+1$ that is repeated at least three times, where $m+1\in S$, for all $m \ge 0$. Thus, the generating function in this case is
\begin{align*}
&\sum_{\substack{m\ge0\\ m+1 \in S } }  \left(q^{3(m+1)}+q^{4(m+1)}+\cdots\right)
\prod_{\substack{i\in S\\ i\ne m+1}}
\left(1+q^{i}+q^{2
i}+\cdots\right)\\
&=\sum_{\substack{m\ge0\\ m+1 \in S } } \left(\dfrac{q^{3(m+1)}}{1-q^{m+1}}\right)
\prod_{\substack{i\in S\\ i\ne m+1}}
\left( \dfrac{1}{1-q^{i}}\right)\\
&=\sum_{\substack{m\ge0\\ m+1 \in S } } \left(\dfrac{q^{3(m+1)}}{1-q^{m+1}}\right)
\left( \frac{(-q^{2r+1};q^{2r+1})_\infty}{(q^2;q^2)_\infty}(1-q^{m+1}) \right)\\
&=\frac{(-q^{2r+1};q^{2r+1})_\infty}{(q^2;q^2)_\infty} \sum_{\substack{m\ge0\\ (m+1) \in S } } q^{3(m+1)}\\
&=\dfrac{(-q^{2r+1};q^{2r+1})_\infty}{(q^2;q^2)_\infty}
\left(
\frac{q^{6r+3}+ \sum_{j =1 }^{2r+1}q^{6j} }{1-q^{12r+6}}
\right).
\end{align*}
By adding all the generating functions derived above, we obtain the requisite expression for the generating function of $b_3^{(r)}(n)$.
\end{proof}
%%%%%%%%%%%%%%%%%%%%%%%%%%%%%%%%%%%%%
\section{Proof of Asymptotic formulas for $a_{\ell}^{(r)}(n)$ and $b_{\ell}^{(r)}(n)$}\label{p1_sec3}
In this section, we prove Theorems~\ref{p1_thm8}, \ref{p1_thm9} and \ref{p1_thm10}. We prove the asymptotic formulas using Wright's circle method.
\subsection{Wright's circle method}\label{p1_subsec3.1}
The circle method is a fundamental technique in the analytic theory of integer partitions. This method originates from the seminal work of Hardy and Ramanujan on $p(n)$ \cite{Hardy_Ramanujan}, in which they established that
\[
p(n) \sim \frac{1}{4n\sqrt{3}}e^{\pi\sqrt{\frac{2n}{3}}}
\]
as $n\to\infty$. Subsequently, Rademacher \cite{Rademacher} refined this approach to derive an exact formula for
$p(n)$. The presence of non-modular terms in the generating
functions associated with $a_\ell^{(r)}(n)$ and  $b_\ell^{(r)}(n)$ for $\ell= 1,2,3$, such as Lambert series and rational functions, precludes the direct application of Rademacher's approach to obtain exact formulas in these cases.
\par
To overcome this difficulty, we employ a modified version of the circle method due to Wright \cite{Wright}. We use the formulation of this method as developed by Ngo and Rhoades \cite{Ngo_Rhoades}, who derived asymptotic formulas for the coefficients $c(n)$ of the products
\[
L(q)\xi(q),
\]
where $L(q)$ and $\xi(q)$ are certain analytic functions of
$q\in\mathbb{C}$ with $|q|<1$ and $q\notin\mathbb{R}_{\le 0}$. Informally, $\xi(q)$ must have a ``main’’ exponential singularity as $q\to1$ and $L(q)$ must have a polynomial behavior as
$q\to 1$.
\par
Throughout the remainder of this article, let $q = e^{-z}$, where $z = x + iy$ and $q \notin \mathbb{R}_{\le 0}$, with $x > 0$ and $|y| < \pi$.

The setup is as follows, since $L(q)$ and $\xi(q)$ are two analytic functions so $L(q)\xi(q)$ is also analytic for
$|q|<1$ and $q\notin\mathbb{R}_{\le 0}$. We are interested in the
asymptotic expansion of
\begin{equation}
c(n)=\frac{1}{2\pi i}
\int_{\mathcal{C}}
\frac{L(q)\xi(q)}{q^{n+1}}dq 
\notag
\end{equation}
  as  $n\to\infty$, where $\mathcal{C}$ denotes a circle centered at the origin with radius less than one, traversed by $q$ in the counterclockwise direction. The above integral arises from the generating functions described in Section~\ref{p1_sec2}.
\par
To state the version of Wright's circle method as presented in~\cite{Ngo_Rhoades}, we use the following notations introduced by Ballantine et al. \cite{Ballantine_2023}. Let $0 < \delta < \frac{\pi}{2}$ and define
\[
D_{\delta}:=\{z\in\mathbb{C}: |\arg(z)|< \frac{\pi}{2} -\delta\}.
\]
Equivalently, if we let $\Delta=\cot(\delta)$, then
\[
D_{\delta}:=\{x+iy\in\mathbb{C}: 0\le |y|<\Delta x\}.
\]
For fixed $c>0$ and integer $n\ge 1$, we let $\mathcal{C}=\mathcal{C}(c,n)$ be the
circle in the complex plane centered at zero with radius $|q|=e^{-x}$, where
$x=\frac{c}{\sqrt{n}}.$ For fixed $\delta$, we let
\[
\mathcal{C}_1:=D_{\delta}\cap\mathcal{C}
\qquad\text{and}\qquad
\mathcal{C}_2:=\mathcal{C}\setminus\mathcal{C}_1.
\]
We call $\mathcal{C}_1$ the \textit{major arc} and $\mathcal{C}_2$ the \textit{minor arc}.
We will interchangeably use Landau's big-$O$ and Vinogradov's $\ll$ notation in this section. Recall that for two complex-valued functions $f(z)$ and $g(z)$, we write $f = O(g)$ or $f \ll g$ in a region if and only if there exists a constant $C > 0$ such that
\[
|f(z)|\le C|g(z)|
\]
for all $z$ in that region.

Let $L(q),\xi(q)$ be analytic functions for $q=e^{-z}$ in the unit disk, and fix
$\delta$ and $c$ as above. Following the framework established by Ngo and Rhoades, we define four hypotheses  concerning the asymptotic properties of these functions. In (H1)–(H4) below, the notation
$\ll_{\delta}$ and $O_{\delta}$ indicate that the implied constants depend on
$\delta$.
\begin{itemize}
    \item[(H1)] For every positive integer $k$, as $z\to 0$ in the cone
    $D_{\delta}$ we have
    \[
    L(q)
    =
    \frac{1}{z^B}
    \left(
        \sum_{s=0}^{k-1}\alpha_s z^s
        + O_{\delta}\left(z^k\right)
    \right),
    \]
    where $B\in\mathbb{R}$ and $\alpha_s\in\mathbb{C}$.

    \item[(H2)] As $z\to 0$ in the cone $D_{\delta}$ we have
    \[
    \xi(q)
    =
    K z^{\beta} e^{\frac{A}{z}}
    \left(
        1+O_{\delta}\left(e^{-\frac{\gamma}{z}}\right)
    \right),
    \]
    where $A:=c^2$, and $K,\beta\ge 0$ and $\gamma>0$.
\item[(H3)] As $z \to 0$ in the region
$\frac{\pi}{2}-\delta \le |\arg(z)| < \frac{\pi}{2},$ we have
\[
|L(q)| \ll_{\delta} |z|^{-C},
\]
where $C>0$.

\item[(H4)] As $z \to 0$ in the region $ \frac{\pi}{2}-\delta \le |\arg(z)| < \frac{\pi}{2},$ we have
\[
|\xi(q)| \ll_{\delta} \xi(|q|)
e^{-\frac{\delta'}{\operatorname{Re}(z)}},
\]
where $\delta'=\delta'(\delta)>0$.
\end{itemize}
Note that, in (H2), Ngo and Rhoades assumed $\gamma>c^2$, but this is not strictly necessary when establishing Proposition \ref{p1_prop1}, although the scenario where $\gamma>c^2$ arises naturally if $\xi$ is modular. In addition, (H1) and (H2) put restrictions on the
asymptotic behavior of $L(q)\xi(q)$ on the major arc, and (H3) and (H4) put restrictions on the minor arc.
Ngo and Rhoades proved the following result from these hypotheses.
\begin{proposition}\label{p1_prop1}\cite[Proposition 1.8]{Ngo_Rhoades} 
   Assume the hypotheses above. Then, as $n\to\infty$, we
have for any $N\in\mathbb{Z}^{+}$ that
\begin{align}
c(n)
= K e^{2\sqrt{An}} n^{\frac{1}{4}(2B-2\beta-3)}
\left(
\sum_{r=0}^{N-1} p_r n^{-\frac{r}{2}}
+O\left(n^{-\frac{N}{2}}\right)
\right),
\notag
\end{align}
where
$p_r:=\sum_{j=0}^{r}\alpha_j c_{j,r-j}$ and
$c_{j,r}
:=
\frac{
\left(-\frac{1}{4\sqrt{A}}\right)^r
\left(\sqrt{A}\right)^{j+\beta-B+\frac{1}{2}}
}{
2\sqrt{\pi}
}
\frac{
\Gamma\left(j+B+\frac{3}{2}+r\right)
}{
r!\Gamma\left(j+B+\frac{3}{2}-r\right)
}.$
 
\end{proposition}
In the context of Proposition \ref{p1_prop1}, it should be emphasized that the specific numerical values of the constants $\delta, \gamma, C, \delta'$ involved do not play a direct role in the asymptotic expression. As a result, no particular attention is devoted to determining these constants with exactness in the analysis conducted in this section.

\subsection{Modular transformation}\label{p1_subsec3.2}

In this section, we present the calculation of the asymptotic growth of $\frac{(-q^{2r+1};q^{2r+1})_\infty}{(q^2;q^2)_\infty}$ on both the major and minor arcs. As established in Section~\ref{p1_sec2}, all the relevant generating functions are expressed as products involving $\frac{(-q^{2r+1};q^{2r+1})_\infty}{(q^2;q^2)_\infty}$ and additional factors. Since $\frac{(-q^{2r+1};q^{2r+1})_\infty}{(q^2;q^2)_\infty}$ constitutes the factor with exponential singularities in all these generating functions, we define 
$$\xi^{(r)}(q):=\frac{(-q^{2r+1};q^{2r+1})_\infty}{(q^2;q^2)_\infty} = \frac{(q^{4r+2};q^{4r+2})_\infty}{(q^2;q^2)_\infty (q^{2r+1};q^{2r+1})_\infty} = \frac{E^{(r)}(q)}{(q^2;q^2)_\infty},
$$
where $E^{(r)}(q):= \frac{(q^{4r+2};q^{4r+2})_\infty}{(q^{2r+1};q^{2r+1})_\infty}$.
Observe that for $q = e^{2\pi i\tau}$, we have the relation $P(e^{2\pi i\tau}) =\frac{e^{\pi i\tau/12}}{\eta(\tau)},$ where $\eta(\tau)$ denotes the Dedekind $\eta$-function and $P(q):= \sum_{n \ge 0} p(n)q^n = \frac{1}{(q;q)_\infty}$ is the Euler’s partition generating function. Note that $P(q)$ satisfies the following transformation
\begin{align}
P(q)=P(e^{-z})=\sqrt{\frac{z}{2\pi}}e^{\left(\frac{\pi^2}{6z}-\frac{z}{24}\right)}
P\left(e^{-\frac{4\pi^2}{z}}\right),
\label{p1_eqn1}
\end{align}
which is an immediate consequence of the modular transformation law
\begin{align*}
    \eta\left(-\frac{1}{\tau}\right) =\sqrt{\frac{\tau}{i}}\eta(\tau)
\end{align*}
satisfied by the Dedekind $\eta$-function.
Using \eqref{p1_eqn1}, we obtain the following transformation law for $\frac{(q^2;q^2)_\infty}{(q;q)_\infty}$:
\begin{align*}
    \frac{(q^2;q^2)_\infty}{(q;q)_\infty}=\frac{1}{\sqrt{2}}e^{\left(\frac{\pi^2}{12z}+\frac{z}{24}\right)}\frac{P\left(e^{-\frac{4\pi^2}{z}}\right)}{P\left(e^{-\frac{2\pi^2}{z}}\right)}. 
\end{align*}
Using this transformation law for $\frac{(q^2;q^2)_\infty}{(q;q)_\infty}$, Jackson and Otgonbayar~\cite{Jackson} established the following asymptotic results for $\frac{(q^2;q^2)_\infty}{(q;q)_\infty}$ on both major and minor arcs.
\begin{lemma}\label{p1_lem1} 
   \cite[Lemma 3.8]{Jackson} Let $\Delta > 0$ be a constant, and suppose $z=x+iy$ satisfies
$0\le |y|<\Delta x$. Then, as $z\to0$ in this region, we have
$$
\frac{(q^2;q^2)_\infty}{(q;q)_\infty} = \frac{1}{\sqrt{2}} e^ {\left( \frac{\pi^2}{12z} + \frac{z}{24} \right)}\left( 1 + O \left( e^{-\frac{2\pi^2}{z}} \right) \right).
$$ 
\end{lemma}

\begin{lemma}
     \cite[Lemma 3.9]{Jackson} \label{p1_lem3} Let $\Delta > 0$ be a positive constant, and suppose $z = x + iy$ satisfies $\Delta x \le |y| < \pi$. Then, as $z \to 0$ in this region, we have
$$
\left|\frac{(q^2;q^2)_\infty}{(q;q)_\infty}\right| \ll e^ {\left( \frac{\pi^2}{12x} \left( \frac{1}{2} + \frac{3}{\pi^2} + \frac{6}{\pi^2 (\Delta^2 + 1)} \right) \right)}.
$$In particular, if $\Delta > 1.45$, then there is a positive constant $\delta'$ such that
$$
\left|\frac{(q^2;q^2)_\infty}{(q;q)_\infty}\right| \ll \left( -e^{-x};e^{-x} \right)_\infty  e^ {\left( -\frac{\delta'}{x} \right)}.
$$
\end{lemma}
We now prove two lemmas on the asymptotic behavior of $\xi^{(r)}(q)$.
\begin{lemma}\label{p1_lem2}
Let $\Delta > 0$ be a constant, and suppose $z=x+iy$ satisfies
$0\le |y|<\Delta x$. Then, as $z\to0$ in this region, we have
    $$ \xi^{(r)}(q) = \frac{\sqrt{z}}{\sqrt{2\pi}} e^ {\left(\frac{(r+1)\pi^2}{6(2r+1)z}  \right)} \left( 1 + O \left( e^{-\frac{2\pi^2}{(2r+1)z}} \right) \right).$$
\end{lemma}
\begin{proof}
Using Lemma \ref{p1_lem1}, as $z\to0$ we obtain
\begin{align}
    E^{(r)}(q) = \frac{1}{\sqrt{2}} e^ {\left( \frac{\pi^2}{12(2r+1)z} + \frac{(2r+1)z}{24} \right)} \left( 1 + O \left( e^{-\frac{2\pi^2}{(2r+1)z}} \right) \right).
    \label{p1_eqn2}
\end{align}
Also, from \eqref{p1_eqn1}, we have
\begin{align*}
\notag P(e^{-z})&= \sqrt{\frac{z}{2\pi}}e^{\frac{\pi^2}{6z} -\frac{z}{24}}\sum_{n\ge 0}p(n)e^{-\frac{4\pi^2n}{z}} \\
&= \sqrt{\frac{z}{2\pi}} e^{\frac{\pi^2}{6z}} \left(1 + O(e^{\frac{-4\pi^2}{z}})\right).
\end{align*}
This yields
\begin{align}
P(e^{-2z}) = \frac{1}{(q^2;q^2)_\infty} = \sqrt{\frac{z}{\pi}}  e^{\left( \frac{\pi^2}{12z}\right)} \left(1 + O(e^{\frac{-2\pi^2}{z}})\right).
\label{p1_eqn3}
\end{align}
Therefore, using \eqref{p1_eqn2} and \eqref{p1_eqn3} as $z\to0$ we have 
\begin{align*}
    \xi^{(r)}(q) &= \frac{(q^{4r+2};q^{4r+2})_\infty}{(q^2;q^2)_\infty (q^{2r+1};q^{2r+1})_\infty}\\ &=  \sqrt{\frac{z}{\pi}}  e^{\left( \frac{\pi^2}{12z}\right)} \left(1 + O(e^{\frac{-2\pi^2}{z}})\right)\frac{1}{\sqrt{2}} e^ {\left( \frac{\pi^2}{12(2r+1)z} + \frac{(2r+1)z}{24} \right)} \left( 1 + O \left( e^{-\frac{2\pi^2}{(2r+1)z}} \right) \right)\\
&=\frac{\sqrt{z}}{\sqrt{2\pi}} e^ {\left( \frac{(r+1)\pi^2}{6(2r+1)z}  \right)} \left( 1 + O \left( e^{-\frac{2\pi^2}{(2r+1)z}} \right) \right).
\end{align*}
This completes the proof of the lemma.
\end{proof}

\begin{lemma}\label{p1_lem4}
Let $\Delta > 1.45$ be a positive constant, and suppose $z = x + iy$ satisfies $\Delta x \le |y| < \pi$. Then, as $z \to 0$ in this region, we have
    $$|\xi^{(r)}(q)| = \left|\frac{E^{(r)}(q)}{(q^2;q^2)_\infty}\right| \ll \xi^{(r)}(e^{-x}) e^ {\left( -\frac{\delta'}{(2r+1)x} \right)} = \xi^{(r)}(e^{-x}) e^ {\left( -\frac{\delta''}{x} \right)},
    $$ where $\delta'' = \frac{\delta'}{(2r+1)}$.
\end{lemma}
\begin{proof}
 Using Lemma  \ref{p1_lem3}, as $z\to0$, we find that
$$
|E^{(r)}(q)| \ll E^{(r)} \left( e^{-(2r+1)x} \right)  e^ {\left( -\frac{\delta'}{(2r+1)x} \right)}.
$$
Also, $\left|\dfrac{1}{(q^2;q^2)_\infty}\right|$ $\le$ $\dfrac{1}{(e^{-2x};e^{-2x})_\infty}$.
Therefore, we have
\begin{align*}
|\xi^{(r)}(q)| = \left|\frac{E^{(r)}(q)}{(q^2;q^2)_\infty}\right| \ll \xi^{(r)}(e^{-x}) e^ {\left( -\frac{\delta'}{(2r+1)x} \right)} = \xi^{(r)}(e^{-x}) e^ {\left( -\frac{\delta''}{x} \right)}.
\end{align*}
  This completes the proof. 
\end{proof}
From Lemma~\ref{p1_lem2} and Lemma~\ref{p1_lem4} we can say that $\xi^{(r)}(q)$ satisfies (H2) and (H4) with the constants $\beta = \frac{1}{2}$, $K = \frac{1}{\sqrt{2\pi}}$ and $A = \frac{(r+1)\pi^2}{6(2r+1)}$.

\subsection{ Estimating rational functions and Lambert series}\label{p1_subsec3.3}
To verify (H1) and (H3) in our situation, we must consider the rational functions and Lambert series that appear in the generating functions for $a_\ell^{(r)}(n)$ and $b_\ell^{(r)}(n) $ for $\ell \in \{ 1,2,3\}$ and $r\ge1$. We define the following functions:
\begin{align*}
L_{a,1}^{(r)}(q) &:= \sum_{m\ge0}   \left(\dfrac{q^{2(m+1)}(1+q^{(2r-1)(m+1)})}{1+q^{(2r+1)(m+1)}} \right), \\
L_{a,2}^{(r)}(q) &:= \sum_{m\ge0} 
\left(
\dfrac{q^{2(m+2)}(1-q^{2(m+1)})(1+q^{(2r-1)(m+2)})}{(1+q^{(2r+1)(m+1)})(1+q^{(2r+1)(m+2)})} + \dfrac{q^{2(m+1)}(1+q^{(2r-1)(m+1)})}{1+q^{(2r+1)(m+1)}} 
\right),\\
L_{a,3}^{(r)}(q) & := \sum_{m\ge0} \left(
 \dfrac{q^{2(m+3)}(1-q^{2(m+1)})(1-q^{2(m+2)})(1+q^{(2r-1)(m+3)})}{(1+q^{(2r+1)(m+1)})(1+q^{(2r+1)(m+2)})(1+q^{(2r+1)(m+3)})}  \right.\\
&\hspace{1 cm}\left.+
\dfrac{q^{2(m+2)}(1-q^{2(m+1)})(1+q^{(2r-1)(m+2)})}{(1+q^{(2r+1)(m+1)})(1+q^{(2r+1)(m+2)})}
+ \dfrac{(q^{4(m+1)} +q^{(2r+1)(m+1)})}{(1+q^{(2r+1)(m+1)}) } \right), \\
L_{b,1}^{(r)}(q) &:= \left(
\frac{q^2 \sum_{k=0}^{2r}q^{2k} +q^{2r+1}}{1-q^{4r+2}} \right), \\
L_{b,2}^{(r)}(q) &:=
\left(
\dfrac{\sum_{k=1}^{2r+1}q^{2k} + q^{2r+1}}{1-q^{4r+2}} + \dfrac{q^{4r+2} + \sum_{k=1}^{2r+1}q^{4k}- q^{4r+1}-q^{4r+3}}{1-q^{8r+4}} 
\right), \\
L_{b,3}^{(r)}(q) &:=
\left(
\frac{q^{2r+1} + \sum_{k=0}^{2r}q^{2k+4}}{1 - q^{4r+2}} +  \frac{q^{4r+2} + \sum_{j=1}^{2r+1}q^{4j} - \sum_{k=0}^{2r}q^{4k+6}}{1 - q^{8r+4}}  \right.\\
&\hspace{1.2cm}\left.+  \frac{2q^{6r+3} - (q^{6r+1} + q^{6r+2} + q^{6r+4} + q^{6r+5}) + \sum_{j=1}^{2r+1}q^{6j}}{1 - q^{12r+6}}
\right). 
\end{align*}
\par
To verify Hypotheses (H1) and (H3) for the functions defined above, we require the following aysmptotic expansion of the Euler-Maclaurin sum formula. Recall that the $n^{th}$ \emph{Bernoulli polynomials} $B_n(x)$ are defined as 
\begin{align*}
\sum_{n\ge 0}B_n(x)\frac{t^n}{n!}:= \frac{te^{tx}}{e^t-1}.
\end{align*}
\begin{theorem}\label{p1_thm17}\cite[Theorem 1.2]{Ingham}
Suppose that 
$0 \le \delta < \frac{\pi}{2}$ and let
$$
D_{\delta}:=\{z\in\mathbb{C}: |\arg(z)|< \frac{\pi}{2} -\delta\}.
$$
Let $f:\mathbb{C}\to\mathbb{C}$ be holomorphic in a domain
containing $D_{\delta}$, so that in particular $f$ is
holomorphic at the origin, and assume that $f$ and all of its
derivatives are of sufficient decay. Then, for $a\in\mathbb{R}$
and nonnegative integer $N$,
$$
\sum_{m\ge 0} f\bigl(z(m+a)\bigr)=\frac{1}{z}\int_0^\infty f(x)dx-\sum_{n=0}^{N-1}\frac{B_{n+1}(a)f^{(n)}(0)}{(n+1)!}z^n+O_N\bigl(z^N\bigr),$$
uniformly, as \(z\to 0\) in \(D_{\delta}\). 
\end{theorem}
We now decompose the relevant series in a manner that will assist in verifying the requisite Hypotheses (H1) to (H4).

\begin{align*}
&\sum_{m\ge0} 
\left(
\dfrac{q^{2(m+2)}(1-q^{2(m+1)})(1+q^{(2r-1)(m+2)})}{(1+q^{(2r+1)(m+1)})(1+q^{(2r+1)(m+2)})} \right) \\
&=
\sum_{m \ge 0}\dfrac{q^2}{1-q^{2r+1}}\left(\dfrac{-q^{2r-1} + (q^{2r-1}+1)q^{2m+2} - q^{4m+4}}{1+q^{(2r+1)(m+1)}}\right.\\
&\hspace{1.2cm}\left.+\dfrac{q^{2r-1} - (q^{2r+1}+q^{2r-1})q^{2m+2} + q^{2r+1}q^{4m+4}}{1+q^{(2r+1)(m+2)}}\right)\\
& =   R_{a,1}(q) \mathcal{L}_{a,1}^{(1)}(q)+ R_{a,2}(q)\mathcal{L}_{a,2}^{(1)}(q) + R_{a,3}(q) \mathcal{L}_{a,1}^{(2)}(q) + R_{a,4}(q) \mathcal{L}_{a,2}^{(2)}(q) + R_{a,5}(q).
\end{align*}
Similarly, we write
\begin{align*}
&\sum_{m \ge 0}\dfrac{q^{2(m+3)}\left(1-q^{2(m+1)}\right)\left(1-q^{2(m+2)}\right)\left(1+q^{(2r-1)(m+3)}\right)}
{\left(1+q^{(2r+1)(m+1)}\right)\left(1+q^{(2r+1)(m+2)}\right)\left(1+q^{(2r+1)(m+3)}\right)}\\
&=\sum_{m \ge 0}\dfrac{1}{1-q^{2r+1}}\left(
\dfrac{-q^{4r+2}+\left(q^{4}+q^{4r+2}+q^{4r+4}\right)q^{2m+2}
-\left(q^{4}+q^{6}+q^{4r+4}\right)q^{4m+4}+q^{6m+12}}
{\left(1-q^{4r+2}\right)\left(1+q^{(2r+1)(m+1)}\right)}\right.\\
&\hspace{.2cm}\left.+\dfrac{q^{4r+2}-\left(q^{2r+5}+q^{4r+2}+q^{4r+4}\right)q^{2m+2}
+\left(q^{2r+5}+q^{2r+7}+q^{4r+4}\right)q^{4m+4}-q^{2r+6m+13}}
{\left(1-q^{2r+1}\right)\left(1+q^{(2r+1)(m+2)}\right)}\right.\\
&\hspace{.2cm}\left.+\dfrac{-q^{6r+3}+\left(q^{6r+3}+q^{6r+5}+q^{6r+7}\right)q^{2m+2}
-\left(q^{6r+5}+q^{6r+7}+q^{6r+9}\right)q^{4m+4}+q^{6r+6m+15}}
{\left(1-q^{4r+2}\right)\left(1+q^{(2r+1)(m+3)}\right)}\right)
\\
&=R_{a,6}(q)\mathcal{L}_{a,1}^{(1)}(q)
+R_{a,7}(q)\mathcal{L}_{a,2}^{(1)}(q)
+R_{a,8}(q)\mathcal{L}_{a,3}^{(1)}(q)+R_{a,9}(q)\mathcal{L}_{a,1}^{(2)}(q)
+R_{a,10}(q)\mathcal{L}_{a,2}^{(2)}(q)\\
&\hspace{.2cm}+R_{a,11}(q)\mathcal{L}_{a,3}^{(2)}(q)+R_{a,12}(q)\mathcal{L}_{a,1}^{(3)}(q)
+R_{a,13}(q)\mathcal{L}_{a,2}^{(3)}(q)
+R_{a,14}(q)\mathcal{L}_{a,3}^{(3)}(q)
+R_{a,15}(q),
\end{align*}
\text{where}
\begin{align*}
\mathcal{L}_{a,1}^{(i)}(q) &:= \sum_{m\ge0}\dfrac{q^{2(m+i)}}{1+q^{(2r+1)(m+i)}},
& \mathcal{L}_{a,2}^{(i)}(q) &:= \sum_{m\ge0}\dfrac{q^{4(m+i)}}{1+q^{(2r+1)(m+i)}}, \\
\mathcal{L}_{a,3}^{(i)}(q) &:= \sum_{m\ge0}\dfrac{q^{6(m+i)}}{1+q^{(2r+1)(m+i)}},
& R_{a,1}(q) &:= \dfrac{q^{2r+1}+q^{2}}{1-q^{2r+1}}, \\
R_{a,2}(q) &:= \dfrac{-q^{2}}{1-q^{2r+1}},
& R_{a,3}(q) &:= \dfrac{-\left(q^{2r+1}+q^{2r-1}\right)}{1-q^{2r+1}}, \\
R_{a,4}(q) &:= \dfrac{q^{2r-1}}{1-q^{2r+1}},
& R_{a,5}(q) &:= \dfrac{q^{2(2r+1)}}{1-q^{2(2r+1)}}, \\
R_{a,6}(q) &:= \dfrac{q^{4}+q^{4r+2}+q^{4r+4}}{\left(1-q^{2r+1}\right)\left(1-q^{4r+2}\right)},
& R_{a,7}(q) &:= \dfrac{-\left(q^{4}+q^{6}+q^{4r+4}\right)}{\left(1-q^{2r+1}\right)\left(1-q^{4r+2}\right)}, \\
R_{a,8}(q) &:= \dfrac{q^{6}}{\left(1-q^{2r+1}\right)\left(1-q^{4r+2}\right)},
& R_{a,9}(q) &:= -\dfrac{q^{2r+3}+q^{4r}+q^{4r+2}}{\left(1-q^{2r+1}\right)^{2}}, \\
R_{a,10}(q) &:= \dfrac{q^{2r+1}+q^{2r+3}+q^{4r}}{\left(1-q^{2r+1}\right)^{2}},
& R_{a,11}(q) &:= -\dfrac{q^{2r+1}}{\left(1-q^{2r+1}\right)^{2}}, \\
R_{a,12}(q) &:= \dfrac{q^{6r+3}+q^{6r-1}+q^{6r+1}}{\left(1-q^{2r+1}\right)\left(1-q^{4r+2}\right)},
& R_{a,13}(q) &:= -\dfrac{q^{6r-1}+q^{6r+1}+q^{6r-3}}{\left(1-q^{2r+1}\right)\left(1-q^{4r+2}\right)}, \\
R_{a,14}(q) &:= \dfrac{q^{6r-3}}{\left(1-q^{2r+1}\right)\left(1-q^{4r+2}\right)},
& R_{a,15}(q) &:= \dfrac{q^{6r+3}\left(1-q^{6r+3}\right)}{\left(1-q^{4r+2}\right)\left(1-q^{8r+4}\right)}.
\end{align*}
We begin by estimating the Lambert series that arises in our generating function, by applying Theorem~\ref{p1_thm17} and obtain the following results. Note that $D_0^{(r)}, D_1^{(r)}, D_2^{(r)},$ and $D_3^{(r)}$ are as defined in Section~\ref{Subsection-Asymp}.
\begin{lemma}\label{p1_lem5}
     Let $\Delta > 0$ be a constant and suppose that $z=x+iy$ satisfies
$0\le |y|<\Delta x$. Then, as $z\to0$ in this region, we have
    \begin{align*}
L_{a,1}^{(r)}(q) = \dfrac{1}{z}\left( \dfrac{D_0^{(r)}}{2r+1} - \dfrac{1}{2}z + \dfrac{1}{12}z^{2} + O(z^{4}) \right).
\end{align*}
\end{lemma}
\begin{lemma}\label{p1_lem6}
  Let $\Delta > 0$ be a constant and suppose that $z=x+iy$ satisfies
$0\le |y|<\Delta x$. Then, as $z\to0$ in this region, we have
   
\begin{align*}
\mathcal{L}_{a,1}^{(1)}(q) &= \dfrac{1}{z}\left(\dfrac{D_1^{(r)}}{2(2r+1)}-\dfrac{z}{4}-\dfrac{2r-3}{48}z^{2}+O(z^{3})\right), \\
\mathcal{L}_{a,1}^{(2)}(q) &= \dfrac{1}{z}\left(\dfrac{D_1^{(r)}}{2(2r+1)}-\dfrac{3z}{4}-\dfrac{13(2r-3)}{48}z^{2} +O(z^{3})\right), \\
\mathcal{L}_{a,1}^{(3)}(q) &= \dfrac{1}{z}\left(\dfrac{D_1^{(r)}}{2(2r+1)}-\dfrac{5z}{4}-\dfrac{37(2r-3)}{48}z^{2}+O(z^{3})\right).
\end{align*}
\end{lemma}

\begin{lemma}\label{p1_lem7}
  Let $\Delta > 0$ be a constant and suppose that $z=x+iy$ satisfies
$0\le |y|<\Delta x$. Then, as $z\to0$ in this region, we have

\begin{align*}
\mathcal{L}_{a,2}^{(1)}(q) &= \dfrac{1}{z}\left(\dfrac{D_2^{(r)}}{2(2r+1)}-\dfrac{z}{4}-\dfrac{2r-7}{48}z^{2}+O(z^{3})\right), \\
\mathcal{L}_{a,2}^{(2)}(q) &= \dfrac{1}{z}\left(\dfrac{D_2^{(r)}}{2(2r+1)}-\dfrac{3z}{4}-\dfrac{13(2r-7)}{48}z^{2} +O(z^{3})\right), \\
\mathcal{L}_{a,2}^{(3)}(q) &= \dfrac{1}{z}\left(\dfrac{D_2^{(r)}}{2(2r+1)}-\dfrac{5z}{4}-\dfrac{37(2r-7)}{48}z^{2} +O(z^{3})\right).
\end{align*}
\end{lemma}

\begin{lemma}\label{p1_lem8}
   Let $\Delta > 0$ be a constant and suppose that $z=x+iy$ satisfies
$0\le |y|<\Delta x$. Then, as $z\to0$ in this region, we have
\begin{align*}
\mathcal{L}_{a,3}^{(1)}(q) &= \dfrac{1}{z}\left(\dfrac{D_3^{(r)}}{2(2r+1)}-\dfrac{z}{4}-\dfrac{2r-11}{48}z^{2}+O(z^{3})\right), \\
\mathcal{L}_{a,3}^{(2)}(q) &= \dfrac{1}{z}\left(\dfrac{D_3^{(r)}}{2(2r+1)}-\dfrac{3z}{4}-\dfrac{13(2r-11)}{48}z^{2}+O(z^{3})\right), \\
\mathcal{L}_{a,3}^{(3)}(q) &= \dfrac{1}{z}\left(\dfrac{D_3^{(r)}}{2(2r+1)}-\dfrac{5z}{4}-\dfrac{37(2r-11)}{48}z^{2} +O(z^{3})\right).
\end{align*}
   \end{lemma}

   \begin{lemma}\label{p1_lem9}  Let $\Delta > 0$ be a constant and suppose that $z=x+iy$ satisfies
$0\le |y|<\Delta x$. Then, as $z\to0$ in this region, we have
       \begin{align*} \sum_{m \ge 0}\dfrac{q^{4(m+1)}+q^{(2r+1)(m+1)})}{1+q^{(2r+1)(m+1)}} = \dfrac{1}{z}\left(\dfrac{2\ln 2+D_2^{(r)}}{2(2r+1)}
-\dfrac{z}{2}+\dfrac{z^{2}}{6}+O(z^{4})\right).
\end{align*}
   \end{lemma}
 Next, we examine the Laurent series expansions of $R_{a,i}(q)$ for all $i \in \{1,2\ldots,15\}$ and using Theorem~\ref{p1_thm17} and obtain the following result.
\begin{lemma}\label{p1_lem10}
 As $z\to 0$, we have 
    \begin{align*}
        R_{a,1}(q) = \dfrac{1}{z} \left( \dfrac{2}{2r+1} - \dfrac{2}{2r+1} z + \left( \dfrac{4r^2 - 8r + 7}{6(2r+1)} \right) z^2  + O(z^3) \right).
    \end{align*}
\end{lemma}
\begin{lemma}\label{p1_lem11}
 As $z \to 0 , $ we have  
    \begin{align*}
         R_{a,2}(q) = \dfrac{1}{z} \left( -\dfrac{1}{2r+1} + \dfrac{3-2r}{2(2r+1)} z - \dfrac{4r^2 - 20r + 13}{12(2r+1)} z^2  + O(z^3) \right).
    \end{align*}
\end{lemma}
\begin{lemma}\label{p1_lem12}
 As $z \to 0$, we have 
    \begin{align*}
         R_{a,3}(q) = \dfrac{1}{z} \left( -\dfrac{2}{2r+1} + \dfrac{2r-1}{2r+1} z - \dfrac{4r^2 - 8r + 7}{6(2r+1)} z^2  + O(z^3) \right).
    \end{align*}
\end{lemma}
\begin{lemma}\label{p1_lem13}
 As $z \to 0$, we have 
    \begin{align*}
        R_{a,4}(q) = \dfrac{1}{z} \left( \dfrac{1}{2r+1} + \dfrac{3-2r}{2(2r+1)} z + \dfrac{4r^2 - 20r + 13}{12(2r+1)} z^2  + O(z^3) \right).
    \end{align*}
\end{lemma}
\begin{lemma}\label{p1_lem14}
 As $z \to 0$, we have 
    \begin{align*}
       R_{a,5}(q) = \dfrac{1}{z} \left( \dfrac{1}{2(2r+1)} - \dfrac{z}{2} + \dfrac{(2r+1)z^2}{6}  + O(z^3) \right).
    \end{align*}
\end{lemma}
\begin{corollary}\label{p1_cor1}
    Let $\Delta > 0$ be a constant and suppose that $z=x+iy$ satisfies
$0\le |y|<\Delta x$. Then, as $z\to0$ in this region, we have
\begin{align*}
L_{a,2}^{(r)}(q)= &\dfrac{1}{z}\left(\dfrac{(2r-3)\left(D_1^{(r)}-D_2^{(r)}\right)+2(2r+1)\left(D_0^{(r)}+1\right)}{2(2r+1)^{2}}
+ O(z) \right).
\end{align*} 
\end{corollary}
\begin{proof} 
The proof readily follows from Lemmas~\ref{p1_lem5} -- \ref{p1_lem7} and  Lemmas~\ref{p1_lem10} -- \ref{p1_lem14}.
\end{proof} 
\begin{lemma}\label{p1_lem15}
 As $z \to 0$, we have 
    \begin{align*}
R_{a,6}(q)
&= \dfrac{1}{z^{2}}\left(\dfrac{3}{2(2r+1)^{2}}+\dfrac{2r-11}{4(2r+1)^{2}}z+\dfrac{12r^{2}-28r+23}{8(2r+1)^{2}}z^{2}  +O(z^{3})\right).
\end{align*} 
\end{lemma}

\begin{lemma}\label{p1_lem16}
 As $z \to 0$, we have 
   \begin{align*}
R_{a,7}(q)
&= \dfrac{1}{z^{2}}\left(-\dfrac{3}{2(2r+1)^{2}}-\dfrac{10r-19}{4(2r+1)^{2}}z-\dfrac{7(2r-3)^{2}}{8(2r+1)^{2}}z^{2} +O(z^{3})\right).
\end{align*} 
\end{lemma}

\begin{lemma}\label{p1_lem17}
 As $z \to 0$, we have 
    \begin{align*}
R_{a,8}(q)
&= \dfrac{1}{z^{2}}\left(\dfrac{1}{2(2r+1)^{2}}+\dfrac{3(2r-3)}{4(2r+1)^{2}}z+\dfrac{44r^{2}-172r+119}{24(2r+1)^{2}}z^{2} +O(z^{3})\right).
\end{align*} 
\end{lemma}

\begin{lemma}\label{p1_lem18}
 As $z \to 0$, we have 
    \begin{align*}
R_{a,9}(q)
&= \dfrac{1}{z^{2}}\left(-\dfrac{3}{(2r+1)^{2}}+\dfrac{2}{2r+1}z-\dfrac{12r^{2}-4r+11}{4(2r+1)^{2}}z^{2}+O(z^{3})\right).
\end{align*} 
\end{lemma}

\begin{lemma}\label{p1_lem19}
 As $z \to 0$, we have 
    \begin{align*}
R_{a,10}(q)
&= \dfrac{1}{z^{2}}\left(\dfrac{3}{(2r+1)^{2}}-\dfrac{1}{2r+1}z+\dfrac{(2r-3)^{2}}{4(2r+1)^{2}}z^{2} +O(z^{3})\right).
\end{align*} 
\end{lemma}

\begin{lemma}\label{p1_lem20}
 As $z \to 0$, we have 
    \begin{align*}
R_{a,11}(q)
&= \dfrac{1}{z^{2}}\left(-\dfrac{1}{(2r+1)^{2}}+\dfrac{1}{12}z^{2}+O(z^{3})\right).
\end{align*} 
\end{lemma}

\begin{lemma}\label{p1_lem21}
 As $z \to 0$, we have 
    \begin{align*}
R_{a,12}(q)
&= \dfrac{1}{z^{2}}\left(\dfrac{3}{2(2r+1)^{2}}+\dfrac{3(1-6r)}{4(2r+1)^{2}}z+\dfrac{44r^{2}-28r+15}{8(2r+1)^{2}}z^{2} +O(z^{3})\right).
\end{align*} 
\end{lemma}

\begin{lemma}\label{p1_lem22}
 As $z \to 0$, we have 
    \begin{align*}
R_{a,13}(q)
&= \dfrac{1}{z^{2}}\left(-\dfrac{3}{2(2r+1)^{2}}-\dfrac{3(5-6r)}{4(2r+1)^{2}}z-\dfrac{44r^{2}-100r+51}{8(2r+1)^{2}}z^{2}+O(z^{3})\right).
\end{align*} 
\end{lemma}

\begin{lemma}\label{p1_lem23}
 As $z \to 0$, we have 
    \begin{align*}
R_{a,14}(q)
&= \dfrac{1}{z^{2}}\left(\dfrac{1}{2(2r+1)^{2}}-\dfrac{3(2r-3)}{4(2r+1)^{2}}z+\dfrac{44r^{2}-172r+119}{24(2r+1)^{2}}z^{2} +O(z^{3})\right).
\end{align*} 
\end{lemma}

\begin{lemma}\label{p1_lem24}
 As $z \to 0$, we have 
    \begin{align*}
R_{a,15}(q)
&= \dfrac{1}{z}\left(\dfrac{3}{8(2r+1)}-\dfrac{9}{16}z+\dfrac{2r+1}{4}z^{2}+O(z^{3})\right).
\end{align*} 

\end{lemma}
\begin{corollary}\label{p1_cor2}
   Let $\Delta > 0$ be a constant and suppose that $z=x+iy$ satisfies
$0\le |y|<\Delta x$. Then, as $z\to0$ in this region, we have
\begin{align*}
L_{a,3}^{(r)}(q) = \frac{1}{z}\Bigg(&\frac{(3+2\ln 2)(2r+1)^{2}+(8r^{2}-10r-1)D_1^{(r)}+(-8r^{2}+28r-8)D_2^{(r)}}{2(2r+1)^{3}} \\
&+\frac{(4r^{2}-14r+10)D_3^{(r)}}{2(2r+1)^{3}} + O(z)\Bigg).
\end{align*}
\end{corollary}
\begin{proof}
The proof readily follows from Lemmas~\ref{p1_lem6}--\ref{p1_lem14} and Lemmas~\ref{p1_lem15}--\ref{p1_lem24}.
\end{proof}
Finally, using Theorem~\ref{p1_thm17} we obtain the following estimates of $L_{b,i}^{(r)}(q)$ and $L_{a,i}^{(r)}(q)$ for $i=1,2,3$. 
\begin{lemma}\label{p1_lem25}
 As $z \to 0$, we have 
    \begin{align*}
        L_{b,1}^{(r)}(q) = \frac{1}{z} \left( \frac{r+1}{2r+1} + O(z) \right).
    \end{align*}
\end{lemma}

\begin{lemma}\label{p1_lem26}
 As $z \to 0$, we have 
    \begin{align*}
        L_{b,2}^{(r)}(q) = \frac{1}{z} \left( \frac{3r + 2}{2(2r + 1)} + O(z)  \right). 
    \end{align*}
\end{lemma}

\begin{lemma}\label{p1_lem27}
 As $z \to 0$, we have 
    \begin{align*}
        L_{b,3}^{(r)}(q) = \frac{1}{z} \left( \frac{16r + 13}{12(2r+1)} + O(z) \right). 
    \end{align*}
\end{lemma}
\begin{lemma}\label{p1_lem28} Suppose $z=x+iy$. For $x>0$ and $0\le|y|<\pi$ we have
$$ \left|\sum_{m\ge0}
\frac{q^{m+1}}{1+q^{(2r+1)(m+1)}} \right|<
\dfrac{1}{x^2}.$$ 
\end{lemma}
\begin{proof}
We have 
\begin{align*}
\left|\sum_{m\ge0}
\frac{q^{m+1}}{1+q^{(2r+1)(m+1)}}\right| 
&\le \sum_{m\ge1}
\left|\frac{q^m}{1+q^{(2r+1)m}}\right| \\
&\le \sum_{m\ge1}
\frac{e^{-mx}}{1-e^{-(2r+1)mx}} \\
&= \sum_{m\ge1} e^{-mx}
\sum_{k\ge0}e^{-(2r+1)mkx} \\
&= \sum_{m\ge1}\sum_{k\ge0}
e^{-(1+(2r+1)k)mx} \\
&= \sum_{n\ge1}
\sigma_{1,2r+1}(n)e^{-nx} \\
&\le \sum_{n\ge1}ne^{-nx}
=\frac{e^x}{(e^x-1)^2}
<\frac{1}{x^2}.
\end{align*}
In the above, $\sigma_{1,2r+1}(n) :=|\{d\in\mathbb{N}:d\mid n,\,
d\equiv1\pmod{2r+1}\}|$.
\end{proof}

\begin{corollary}\label{p1_cor3} Suppose $z=x+iy$ satisfies $x>0, 0\le|y|<\pi$. Then, we have
$$|L_{a,1}^{(r)}(q)|<\dfrac{2}{x^2}.$$     
\end{corollary}
\begin{proof}
We have
   \begin{align*}
|L_{a,1}^{(r)}(q)|
&\le \sum_{m\ge0}
\left| \frac{q^{2(m+1)}\left(1+q^{(2r-1)(m+1)}\right)}
     {1+q^{(2r+1)(m+1)}} \right|.
\end{align*}
Since $|1+ q^{(2r-1)(m+1)}| \le 1+|q|^{(2r-1)(m+1)} \le 2$, we obtain
\begin{align*}
|L_{a,1}^{(r)}(q)|
     &\le 2\sum_{m\ge0}
\left|\frac{q^{2(m+1)}}{1+q^{(2r+1)(m+1)}} \right|\\
&\le 2\sum_{m\ge0}|q^{m+1}|\left|
\frac{q^{m+1}}{1+q^{(2r+1)(m+1)}} \right|\\
&<\dfrac{2}{x^2}. 
\end{align*}
The last inequality follows from Lemma~\ref{p1_lem28}.
\end{proof}
Furthermore, on the minor arcs, Corollary~\ref{p1_cor3} shows that the generating function for $a_1^{(r)}(n)$ satisfies Hypothesis (H3).
\begin{corollary}\label{p1_cor4}
    Suppose $z=x+iy$ satisfies $x>0, 0\le|y|<\pi$. As $z\to 0$,  we have
$$|L_{a,2}^{(r)}(q)|<\dfrac{6}{x^4}.$$ 
\end{corollary}
\begin{proof}
We have 
    \begin{align*}
        |L_{a,2}^{(r)}(q)| &\le \sum_{m\ge0} 
\left|
\frac{q^{2(m+2)} \left(1-q^{2(m+1)}\right) \left(1+q^{(2r-1)(m+2)}\right)}{\left(1+q^{(2r+1)(m+1)}\right) \left(1+q^{(2r+1)(m+2)}\right)}\right|
+  \sum_{m\ge0} \left|
\frac{q^{2(m+1)} \left(1+q^{(2r-1)(m+1)}\right)}{1+q^{(2r+1)(m+1)}}
\right|. \\
&<4 \sum_{m\ge0} 
\left|
\frac{q^{2(m+2)}}{\left(1+q^{(2r+1)(m+1)}\right) \left(1+q^{(2r+1)(m+2)}\right)}\right| + \dfrac{2}{x^2}\qquad \text{ (by Corollary~\ref{p1_cor3})} \\
&\le  4 \left (\sum_{m\ge0} 
\left|
\frac{q^{m+2}}{\left(1+q^{(2r+1)(m+1)}\right)}\right|  \sum_{m\ge0} 
\left|
\frac{q^{m+2}}{\left(1+q^{(2r+1)(m+2)}\right)}\right| \right) + \dfrac{2}{x^2} \\
&\le  4 \left (\sum_{m\ge0} 
\left|
\frac{q^{m+1}}{\left(1+q^{(2r+1)(m+1)}\right)}\right| \right)^2 + \dfrac{2}{x^2} \\
&\le \dfrac{4}{x^4} + \dfrac{2}{x^2}  \le \dfrac{6}{x^4}.\qquad \qquad \text{ (by Lemma~\ref{p1_lem28} and the fact that $z\to 0$)}
\end{align*}
\end{proof}
From Corollary~\ref{p1_cor4}, it follows that, on the minor arcs, the generating function for $a_2^{(r)}(n)$ satisfies Hypothesis (H3).
\begin{lemma}\label{p1_lem29}  Suppose $z=x+iy$. For $x>0$ and $0\le|y|<\pi$, as $z\to 0$, we have
$$\left|\sum_{m\ge0}
\frac{
q^{m+3}
}{
\left(1+q^{(2r+1)(m+1)}\right)
\left(1+q^{(2r+1)(m+2)}\right)
}\right|= O\left(\dfrac{1}{x^3}\right).$$  
\end{lemma}
\begin{proof}
We have
    \begin{align*}
      &\left|\sum_{m\ge0}
\frac{
q^{m+3}
}{
\left(1+q^{(2r+1)(m+1)}\right)
\left(1+q^{(2r+1)(m+2)}\right)
}\right|  \\
&\le \sum_{m\ge 1 }
\frac{
e^{-(m+2)x}
}{
\left(1-e^{-(2r+1)mx}\right)
\left(1-e^{-(2r+1)(m+1)x}\right)
} \\
&\le \sum_{m \ge 1 }e ^{-(m+2)x}\left( \sum_{k \ge 0 }e ^{-(2r+1)mkx} \right)\times\left( \sum_{l \ge 0 }e ^{-(2r+1)(m+1)lx} \right) \\
&\le \sum_{m \ge 1 }e ^{-(m+2)x} \sum_{s \ge 0 }(s+1)e ^{-(2r+1)msx} \\
&\le \sum_{m \ge 1 }\sum_{s \ge 0 }(s+1) e^{-((2r+1)s+1)mx}\qquad (\text{ as } e^{-x} <1)\\
    &\le \sum_{m \ge 1 }\sum_{s \ge 0 }((2r+1)s+1) e^{-((2r+1)s+1)mx} \\
&\le \sum_{n \ge 1} \sigma_1(n)e^{-nx}\qquad (\text{ where, }  \sigma_1(n) := \sum_{d|n, d\ge1}d )\\
&\leq \dfrac{1}{2} \left(  \sum_{n \ge 1}n^2e^{-nx} + \sum_{n \ge 1}ne^{-nx} \right) \\
&\le   \dfrac{e^{-x}(1+ e^{-x})}{(1-e^{-x})^3} + \dfrac{1}{2x^2}\\
&=  O\left(\dfrac{1}{x^3}\right) + \dfrac{1}{2x^2}  =  O\left(\dfrac{1}{x^3}\right).
    \end{align*}
This completes the proof.
\end{proof}

\begin{corollary}\label{p1_cor5}
 Suppose $z=x+iy$. For $x>0$ and $0\le|y|<\pi$, as $z\to0$, we have
$$|L_{a,3}^{(r)}(q)|= O\left(\dfrac{1}{x^5}\right).$$     
\end{corollary}
\begin{proof}
We have 
\begin{align}
|L_{a,3}^{(r)}(q)| &\le \sum_{m\ge0}
\left|
\frac{
q^{2(m+3)}
\left(1-q^{2(m+1)}\right)
\left(1-q^{2(m+2)}\right)
\left(1+q^{(2r-1)(m+3)}\right)
}{
\left(1+q^{(2r+1)(m+1)}\right)
\left(1+q^{(2r+1)(m+2)}\right)
\left(1+q^{(2r+1)(m+3)}\right)
}\right| \notag \\ 
&\quad+
    \sum_{m\ge0}
\left|
\frac{
q^{2(m+2)}
\left(1-q^{2(m+1)}\right)
\left(1+q^{(2r-1)(m+2)}\right)
}{
\left(1+q^{(2r+1)(m+1)}\right)
\left(1+q^{(2r+1)(m+2)}\right)
}\right|
+
\sum_{m\ge0}
\left|
\frac{
\left(q^{4(m+1)}
+q^{(2r+1)(m+1)}\right)
}{
1+q^{(2r+1)(m+1)}
}\right| \notag \\
&\le 8 \sum_{m\ge0}
\left|
\frac{
q^{2(m+3)}}{
\left(1+q^{(2r+1)(m+1)}\right)
\left(1+q^{(2r+1)(m+2)}\right)
\left(1+q^{(2r+1)(m+3)}\right)
}\right| \notag \\ 
&\quad+
    4 \sum_{m\ge0}
\left|
\frac{
q^{2(m+2)}
}{
\left(1+q^{(2r+1)(m+1)}\right)
\left(1+q^{(2r+1)(m+2)}\right)
}\right|
+
2\sum_{m\ge0}
\left|
\frac{
q^{m+1}
}{
1+q^{(2r+1)(m+1)}
}\right| \notag \\
&\le 8\left(\sum_{m\ge0}
\left|
\frac{
q^{m+3}
}{
\left(1+q^{(2r+1)(m+1)}\right)
\left(1+q^{(2r+1)(m+2)}\right)
}\right|\right)\left(\sum_{m\ge0}\left|\frac{q^{m+3}}{1+q^{(2r+1)(m+3)}}\right|\right) \notag \\ 
&\quad+4 \sum_{m\ge0}
\left|
\frac{
q^{m+2}
}{
\left(1+q^{(2r+1)(m+1)}\right)
}\right|\sum_{m\ge0}\left|\dfrac{q^{m+2}}{\left(1+q^{(2r+1)(m+2)}\right)}\right|+O\left(\dfrac{1}{x^2}\right) \label{eqn-new-1}\\ 
&=  O\left(\dfrac{1}{x^3}\right) O\left(\dfrac{1}{x^2}\right) + O\left(\dfrac{1}{x^4}\right)+O\left(\dfrac{1}{x^2}\right) =  O\left(\dfrac{1}{x^5}\right). \notag 
\end{align}
We obtain the inequality \eqref{eqn-new-1} from Lemma~\ref{p1_lem28}. The final inequality is obtained using Lemma~\ref{p1_lem28} and Lemma~\ref{p1_lem29}. This completes the proof.
\end{proof}
From Corollary~\ref{p1_cor5}, it follows that, on the minor arcs, the generating function for $a_3^{(r)}(n)$ satisfies Hypothesis (H3).

\subsection{Proof of Theorems~\ref{p1_thm8}, \ref{p1_thm9} and \ref{p1_thm10}}
In this section, we prove Theorems~\ref{p1_thm8}, \ref{p1_thm9}, and \ref{p1_thm10} by systematically verifying that each of the hypotheses required for Proposition~\ref{p1_prop1} is fulfilled by the generating functions of $a_\ell^{(r)}(n)$ and $b_\ell^{(r)}(n)$ for $\ell=1,2,3$.
\begin{proof}[Proof of Theorem~\ref{p1_thm8}]
The function $\xi^{(r)}(q)$ associated with $a_{\ell}^{(r)}(n)$ and  $b_{\ell}^{(r)}(n)$ for $\ell=1,2,3$, satisfies Hypotheses (H2) and (H4), as established in Lemmas~\ref{p1_lem2} and \ref{p1_lem4}, with constants $(\beta, K, A) = \left(\frac{1}{2}, \frac{1}{\sqrt{2\pi}}, \frac{(r+1)\pi^2}{6(2r+1)}\right)$. Additionally, as a result of Lemma~\ref{p1_lem5} and Corollary \ref{p1_cor3} (resp. Lemma~\ref{p1_lem25}), Hypotheses (H1) and (H3) are satisfied by $L_{a,1}^{(r)}(q)$ (resp. $L_{b,1}^{(r)}(q)$), with constants $(B,\alpha_0) = \left(1, \frac{D_0^{(r)}}{2r+1}\right)$ (resp. $(B,\alpha_0) = \left(1, \frac{r+1}{2r+1}\right))$. Now, Theorem~\ref{p1_thm8} readily follows from Proposition~\ref{p1_prop1}.
\end{proof}
\begin{proof}[Proof of Theorem~\ref{p1_thm9}] As a consequence of Corollaries~\ref{p1_cor1} and  \ref{p1_cor4} (resp. Lemma~\ref{p1_lem26}), Hypotheses (H1) and (H3) are  satisfied by $L_{a,2}^{(r)}(q)$ (resp.  $L_{b,2}^{(r)}(q)$) with constants $$(B,\alpha_0) = \left(1,\frac{(2r-3)\left(D_1^{(r)}-D_2^{(r)}\right)+2(2r+1)\left(D_0^{(r)}+1\right)}{2(2r+1)^{2}} \right)$$ (resp. $(B,\alpha_0) =\left(1, \frac{3r+2}{2(2r+1)}\right) $). Hence, Theorem \ref{p1_thm9} readily follows from Proposition~\ref{p1_prop1}.
\end{proof}
\begin{proof}[Proof of Theorem~\ref{p1_thm10}] Similarly, we verify that Hypotheses (H1) and (H3) are satisfied for $L_{a,3}^{(r)}(q)$ (respectively, $L_{b,3}^{(r)}(q)$), with constants $(B,\alpha_0) = \left(1, \frac{\mathcal{D}(r)}{2(2r+1)^{3}}\right)$, where 

\begin{align}\label{new-eqn-2}
     \mathcal{D}(r) &:= (3+2\ln 2)(2r+1)^2+(8r^2-10r-1)D_1^{(r)} \\
&\hspace{1 cm}+  (-8r^2+28r-8)D_2^{(r)}+(4r^2-14r+10)D_3^{(r)}.\notag 
 \end{align}
  $\left(resp. (B,\alpha_0) = \left(1, \frac{16r+13}{12(2r+1)} \right) \right)$, by using Corollaries~\ref{p1_cor2} and \ref{p1_cor5} (resp. Lemma~\ref{p1_lem27}). Hence, Theorem~\ref{p1_thm10} readily follows from Proposition~\ref{p1_prop1}.
\end{proof}
%%%%%%%%%%%%%%%%%%%%%%%%%%%%%%%%%%%%%%%%%%%%%%%%%%%%%%%%%%%%%%
\section{Proof of the main results}\label{p1_sec4}
In this section, we prove our main results Theorems~\ref{p1_thm1.3}--\ref{p1_newthm}.
We first recall the \emph{digamma} function, defined as the logarithmic derivative of $\Gamma$, $$\psi(z) := \frac{\Gamma'(z)}{\Gamma(z)},$$ where $\Gamma$ is defined for $\text{Re}(z)>0$ by
$$\Gamma(z):=\int_{0}^{\infty}e^{-u}u^{z-1}du.$$ 
For $z\in\mathbb{C}$ with $Re(z) > 0$, the \emph{digamma} function has the integral representation:
\begin{align*}
\psi(z)=\int_0^\infty\left(\frac{e^{-u}}{u}-\frac{e^{-zu}}{1-e^{-u}}\right)du.
\end{align*}
The \emph{digamma} function and its derivative have the series representation for
$z\neq-1,-2,\dots$:
\begin{align*}
\psi(z)=-\gamma+\sum_{n\geq 0}
\left(\frac{1}{n+1}-\frac{1}{n+z}\right),
\qquad
\psi'(z)=\sum_{n\geq 0}\frac{1}{(n+z)^2},
\end{align*}
where $\gamma$ is the Euler-Mascheroni constant. From the expression for $\psi'$ it is clear that $\psi$ is strictly increasing on the positive real line.
We now recall the following identities satisfied by the \emph{digamma} function.
\begin{align}
    \psi(z+1)&=\psi(z)+\frac{1}{z},\label{p1_eqn4}\\
    \psi(2z)&=\frac{1}{2}\psi(z)+\frac{1}{2}\psi\left(z+\frac{1}{2}\right)+\ln 2, \label{p1_eqn5}\\
    \psi\left(z+\frac{1}{2}\right)-\psi(z)&=2\sum_{m\ge0}\frac{(-1)^m}{(2z+m)}.\label{p1_eqn6}
\end{align}
Before proving our main results, we define a sequence of ratios to facilitate the subsequent analysis. We denote these ratios by $M_j(r)$, where $j \in \{1, 2, \ldots, 9\}$.
\begin{align*}
    M_1(r) &:= \lim_{n\to\infty}\frac{b_1^{(r)}(n)}{a_1^{(r)}(n)} = \frac{1}{M_1^*(r)}, 
    & M_2(r) &:= \lim_{n\to\infty}\frac{b_2^{(r)}(n)}{a_2^{(r)}(n)}; \\
    M_3(r) &:= \lim_{n\to\infty}\frac{b_3^{(r)}(n)}{a_3^{(r)}(n)}, 
    & M_4(r) &:= \lim_{n\to\infty}\frac{a_2^{(r)}(n)}{a_1^{(r)}(n)}; \\
    M_5(r) &:= \lim_{n\to\infty}\frac{a_2^{(r)}(n)}{a_3^{(r)}(n)}, 
    & M_6(r) &:= \lim_{n\to\infty}\frac{a_3^{(r)}(n)}{a_1^{(r)}(n)}; \\
     M_7(r) &:= \lim_{n\to\infty}\frac{b_2^{(r)}(n)}{b_1^{(r)}(n)}, 
    & M_8(r) &:= \lim_{n\to\infty}\frac{b_2^{(r)}(n)}{b_3^{(r)}(n)}; \\
     M_9(r) &:= \lim_{n\to\infty}\frac{b_3^{(r)}(n)}{b_1^{(r)}(n)}.
    \end{align*}
Our aim is to establish that, as $n \to \infty$, each ratio $M_j(r)$ is greater than one for $j \in \{1, 2, \ldots, 9\}$, except for $M_2(1)$ and $M_3(1)$, which are both less than one. For the remainder of this section, we define $t_k(r):= \frac{k}{2r+1}$.
%%%%%%%%%%%%%%%%%
\begin{proof}[Proof of Theorem~\ref{p1_thm1.3}] From Theorem~\ref{p1_thm8}, it follows that $M_1^*(r) := \frac{D_0^{(r)}}{r+1}$. Applying the Maclaurin series expansion of $\psi(1+x)$ around $x=0$, we obtain $\psi(x+1) = -\gamma + O(x)$ as $x \to 0^+$. Therefore, from \eqref{p1_eqn4}, we have $\psi(x) = -\gamma - \frac{1}{x} + O(x)$ as $x \to 0^+$, which in turn implies that as $r \to \infty$, 
\begin{align*}
 D_0^{(r)} &= \frac{2r+1}{2} + O\left(\frac{1}{r}\right),\\
  M_1^*(r) &=\frac{\frac{2r+1}{2}+O\left(\frac{1}{r}\right)}{r+1}\to 1.  
\end{align*}
Hence, we have $$\displaystyle \lim_{r\to \infty}\left(\lim_{n\to \infty}\frac{b_1^{(r)}(n)}{a_1^{(r)}(n)}\right)=\lim_{r\to \infty}\frac{1}{M_1^{*}(r)}=1.$$
Again using the fact that $\psi(y)-\psi(x)=\sum_{n\geq 0}\left(\frac{1}{n+x}-\frac{1}{n+y}\right)$ we obtain the following
\begin{align*}
M_1^*(r) = \dfrac{2r+1}{r+1} \sum_{n \geq 0} \dfrac{1}{(n(2r+1)+1)(n(2r+1)+2)}.
\end{align*}
We verify that $M_1^*(r+1) - M_1^*(r) > 0$ for all $r \geq 1$, which shows that $M_1^*(r)$ is a strictly increasing sequence for $r\ge1$. In addition, $M_1^*(1) = \frac{\pi}{2\sqrt{3}} < 1$ and $M_1^*(r) \to 1$ as $r\to \infty$.
Therefore, $M_1^*(r) < 1$ for all $r \geq 1$, which implies $M_1(r) > 1$ for all $r \geq 1$. This gives the required bias between $a_{1}^{(r)}(n)$ and $b_{1}^{(r)}(n)$.
\par Next, $M_1(r) = \displaystyle  \lim_{n\to\infty}\frac{b_1^{(r)}(n)}{a_1^{(r)}(n)}>1$ implies that there exists a positive integer $n_0$ such that
$$\left|\frac{b_1^{(r)}(n)}{a_1^{(r)}(n)}-M_1(r)\right|<\frac{M_1(r)-1}{2}$$
for all $n\geq n_0$. Hence,
$$b_{1}^{(r)}(n)-a_{1}^{(r)}(n)=\left(\frac{b_{1}^{(r)}(n)}{a_{1}^{(r)}(n)}-1\right)a_{1}^{(r)}(n)>\left(\frac{M_1(r)-1}{2}\right)a_{1}^{(r)}(n)$$
for all $n\geq n_0$.
From Theorem~\ref{p1_thm8} we have $a_{1}^{(r)}(n) \to \infty$ as $n \to \infty$ for all $r\geq 1$. Hence $b_{1}^{(r)}(n)-a_{1}^{(r)}(n) \to \infty$ as $n \to \infty$. This completes the proof of the theorem.
\end{proof}
%%%%%%%%%%%%%%%%%%%%%%%%%%
\begin{proof}[Proof of Theorem~\ref{p1_thm1.4}] We first define four functions $G_1$, $H_1$, $g_1$, and $h_n$ for $u > 0$ as follows:
\begin{align*}
    &G_1(u):=\psi\left(\dfrac{1}{2}+u\right)-\psi(u)-\psi\left(\dfrac{1}{2}+2u\right)+\psi(2u),\\
    &H_1(u):= \psi(2u)-\psi(u)+1, \\
    &g_1(u) :=\dfrac{t_1(r)}{(u+t_1(r))(u+2t_1(r))},\\
    &h_{n}(u)=\dfrac{u}{(n+u)(n+2u)},
\end{align*}
where $t_1(r)=\frac{1}{2r+1}$. By Theorem~\ref{p1_thm9}, we have
\begin{align*}
    M_2(r) &=\dfrac{(3r+2)(2r+1)}{ 2(2r+1)(D_0^{(r)}+1) + (2r-3)(D_1^{(r)}-D_2^{(r)})}\\
           &= \dfrac{3+t_1(r)}{2t_1(r)((1-4t_1(r))G_1(t_1(r))+2H_1(t_1(r)))}.
\end{align*}
Using the fact that $\displaystyle\psi(y)-\psi(x)=\sum_{n\geq 0}\left(\frac{1}{n+x}-\frac{1}{n+y}\right)$, we obtain
\begin{align*}
    G_1(t_1(r)) &=\sum_{n\ge0}t_1(r)\left(\dfrac{1}{(n+t_1(r))(n+2t_1(r))}-\dfrac{1}{(n+\frac{1}{2}+t_1(r))(n+\frac{1}{2}+2t_1(r))}\right) \\
    &= \sum_{n\ge0}\left(g_1(n)-g_1\left(n+\frac{1}{2}\right)\right) \\
    &= \sum_{ n \ge 0 } \int_{n}^{n+\frac{1}{2}}(-g_1'(u))du <\sum_{n\ge0}\int_{n}^{n+1}(-g_1'(u))du   \\
    &=\int_{0}^{\infty}(-g_1'(u))du=g_1(0)-\lim_{u\to\infty}g_1(u)= g_1(0)=\dfrac{1}{2t_1(r)}.
    \end{align*}  
In the preceding steps, we have used the facts that $g_1'(u)<0$ and $\lim_{u \to \infty}g_1(u)=0$.
\par Similarly, we have
\begin{align*}
    H_1(t_1(r)) &=1+\dfrac{1}{2t_1(r)}+\sum_{n\ge1}\dfrac{t_1(r)}{(n+2t_1(r))(n+t_1(r))} \\
    &=1+\dfrac{1}{2t_1(r)}+\sum_{n\ge1}h_{n}(t_1(r)).
   \end{align*}  
For $0<x\leq\frac{1}{5}$ and $n\ge 1$, it is easy to see that $h_n'(x)=\dfrac{n^2-2x^2}{(n+x)^2(n+2x)^2}>0$. Thus, $h_n(x)$ is an increasing function with respect to $x$ for each $n$, and hence $$\sum_{n \geq 1} h_n(x)\le\sum_{n \geq 1} h_n\left(\frac{1}{5}\right).$$
Now, observe that the function $\frac{5}{(5x+1)(5x+2)}$ is decreasing for $x>0$, and therefore $$\int_{n-1}^n\frac{5}{(5x+1)(5x+2)}dx\ge \frac{5}{(5n+1)(5n+2)} ~\text{for}~ n\ge2.$$ Thus, we have
\begin{align*}
    \sum_{n\ge1}h_n(t_1(r))\le\sum_{n\ge1}h_n\left(\frac{1}{5}\right)&\le \sum_{n=1}^3h_n\left(\frac{1}{5}\right)+\sum_{n\ge4}\int_{n-1}^n\frac{5}{(5x+1)(5x+2)}dx\\
    &\le \frac{5}{42}+\frac{5}{132}+\frac{5}{272}+\int_{3}^\infty\frac{5}{(5x+1)(5x+2)}dx\\
    &\le \frac{5}{42}+\frac{5}{132}+\frac{5}{272}+\ln\left(\frac{17}{16}\right)<\frac{1}{4}.
\end{align*}

This further implies that, for $0<t_1(r)\le\frac{1}{5}$, we have 
\begin{align*}
H_1(t_1(r)) &< \frac{1}{2t_1(r)} + \frac{5}{4},\\
    (1-4t_1(r))G_1(t_1(r))+2H_1(t_1(r)) &<(1-4t_1(r))\dfrac{1}{2t_1(r)}+2\left(\dfrac{1}{2t_1(r)}+\dfrac{5}{4}\right)=\dfrac{3+t_1(r)}{2t_1(r)},
\end{align*}
which yield 
\begin{align*}
    2t_1(r)\left((1-4t_1(r))G_1(t_1(r))+2H_1(t_1(r))\right)<(3+t_1(r)).
\end{align*}
We observe that the denominator of $M_2(r)$ is positive. Therefore, $M_2(r) > 1$ for all  $r \geq 2$. 
\par For the case $r = 1$, we have
\begin{align*}
    M_2(1) = \dfrac{15}{6H_1\left(\frac{1}{3}\right)-G_1\left(\frac{1}{3}\right)}=\dfrac{15}{12+\frac{2\pi}{\sqrt{3}}} <1.
\end{align*} 
This completes the proof of part (i) and part (ii) of the theorem. Now, as shown in the proof of Theorem~\ref{p1_thm1.3}, part (iii) of the theorem readily follows using the fact that $a_{2}^{(r)}(n), b_{2}^{(r)}(n) \to \infty$ as $n \to \infty$ for all $r\geq 1$.
\par 
Finally, we prove part (iv). Using \eqref{p1_eqn5} we obtain
$$D_1^{(r)}-D_2^{(r)}=3\psi\left(\frac{2}{2r+1}\right)-2\psi\left(\frac{4}{2r+1}\right)-\psi\left(\frac{1}{2r+1}\right)+D_0^{(r)}.$$
Also, we know that $\psi(x) = -\gamma - \frac{1}{x} + O(x)$ as $x \to 0^+$. Thus, for sufficiently large $r$, we have
\begin{align*}
    M_2(r)=\dfrac{(3r+2)(2r+1)}{(4r+2)\left(\frac{2r+1}{2}+O\left(\frac{1}{r}\right)\right)+(2r-3)\left(\frac{2r+1}{2}+O\left(\frac{1}{r}\right)\right)}.
 \end{align*}
 Therefore, $M_2(r)\to1$ as $r\to\infty$. This completes the proof.
\end{proof}
%%%%%%%%%%%%%%%%%%%%%%%%
\begin{proof}[Proof of Theorem \ref{p1_thm1.5}]
By Theorem~\ref{p1_thm10}, we have
\begin{align}\label{eqn-new-11}
M_3(r) =\frac{(16r+13)(2r+1)^2}{6 \mathcal{D}(r)},
\end{align}
 where $\mathcal{D}(r)$ is as given in \eqref{new-eqn-2}.
 \par 
To proceed with the proof, we next define the functions $G_2(u)$ and $H_2(r)$ for $u > 0$ and $r \in \mathbb{N}$ as follows.
\begin{align*}
    &G_2(u):=\psi\left(\frac12+u\right)-\psi(1+u)+2\ln 2, \\
    &H_2(r) := \sum_{k=1}^{3} C_k(r) G_2\left(t_k(r)\right),\end{align*}
where $C_1(r) := 8r^2 - 10r - 1$, $C_2(r) := -8r^2 + 28r - 8$, $C_3(r) := 4r^2 - 14r + 10$, and $t_k(r)=\frac{k}{2r+1}$. Now, we find that 
\begin{align}\label{new_eqn1}
\sum_{k=1}^3 C_k(r) &= (2r+1)^2,\\
\label{new_eqn2}
\sum_{k=1}^3 \frac{C_k(r)}{k} &= \frac{4(2r+1)^2 - 9(2r+1)}{3},\\
\label{new_eqn3}
\sum_{k=1}^3 kC_k(r) &= (2r+1)^2 + 12.
\end{align}
Using \eqref{p1_eqn4}, we obtain
\begin{align*}
D_k^{(r)} = \psi\left(\frac{1}{2} + t_k(r)\right) - \psi(t_k(r)) = \frac{(2r+1)}{k} - 2\ln 2 + G_2(t_k(r)),
\end{align*}
for $k \in \{1, 2, 3\}$.
Replacing these expression for $D_k^{(r)}$ in the expression of $\mathcal{D}(r)$, and then using \eqref{new_eqn1} and \eqref{new_eqn2} we obtain 
\begin{align*}
    \mathcal{D}(r) = \dfrac{4(2r+1)^3}{3}+H_2(r).
\end{align*}
Hence, \eqref{eqn-new-11} yields
\begin{align*}
    M_3(r)=\dfrac{8(2r+1)^3+5(2r+1)^2}{8(2r+1)^3+6H_2(r)}.
\end{align*}
For $r=1$, direct computation shows that  $M_3(1) = \frac{87}{2(99-10\pi \sqrt{3})}<1$. This completes the proof of part (i).
\par To prove part (ii), we show that $M_3(r)>1$ for $r\geq 2$. We numerically verify that $M_3(2)>1$. To establish $M_3(r)>1$ for $r\geq 3$, it suffices to show that $0<6H_2(r)<5(2r+1)^2$. Observe that $G_2(0)=0$ and $G_2'$ is decreasing for $u> 0$. Thus, by the Mean Value Theorem, we obtain
\begin{align}
G_2(u) = G_2'(\chi_1) u \leq G_2'(0)u = \frac{\pi^2}{3}u,
\label{p1_eqn7}
\end{align}
for some $\chi_1 \in (0, u)$. Now, combining \eqref{p1_eqn7} and \eqref{new_eqn3}, we obtain
\begin{align*}
H_2(r) \leq \frac{\pi^2}{3} \frac{1}{(2r+1)} \sum_{k=1}^3 kC_k(r) = \frac{\pi^2((2r+1)^2 + 12)}{3(2r+1)}.
\end{align*}
This yields
\begin{align}
6H_2(r)\leq \frac{2\pi^2((2r+1)^2 + 12)}{(2r+1)}.
\label{p1_new-added_eqn1}
\end{align}
It is easy to see that $2\pi^2((2r+1)^2 + 12) < 5(2r+1)^3$ for all $r \geq 3$.
Therefore, from \eqref{p1_new-added_eqn1}, we have $6H_2(r) < 5(2r+1)^2$ for $r \geq 3$. Next, it remains to prove that $H_2(r) > 0$ for all $r \geq 3$ (which ensures the positivity of the denominator of $M_3(r)$). When $r = 3$, we have $C_k(r) > 0$ for $k = 1, 2, 3$, and hence $H_2(3) > 0$ as $G_2(u)>0$ for $u>0$. When $r\ge 4$, we have $C_1(r)>0$, $C_3(r)>0$, but $C_2(r)<0$. Since $G_2''$ is an increasing function, so $G_2''(u)\ge G_2''(0)$ for $u>0$ and integrating this inequality twice from $0$ to $u$, we obtain
\begin{align}\label{new_eqn4}
    G_2'(0)u+\frac{G_2''(0)}{2}u^2\le G_2(u).
\end{align}
Thus, from \eqref{p1_eqn7} and \eqref{new_eqn4}, we obtain
\begin{align*}
    C_2(r)G_2(t_2(r))&\ge C_2(r)G_2'(0)t_2(r),\qquad  \left(\text{ as }C_2(r)<0\right)\\
    C_k(r)G_2(t_k(r))&\ge C_k(r)\left( G_2'(0)t_k(r)+\frac{G_2''(0)}{2}t_k(r)^2\right), \quad \text{for~} k \in \{1,3\}.
\end{align*}
Therefore, from the definition of $H_2(r)$ we find that 
\begin{align*}
    H_2(r)&\ge G_2'(0)\sum_{k=1}^{3}C_k(r)t_k(r)+\frac{G_2''(0)}{2}\left(C_1(r)(t_1(r))^2+C_3(r)(t_3(r))^2\right)\\
    &=\frac{G_2'(0)}{(2r+1)}\sum_{k=1}^3kC_k(r)+\frac{G_2''(0)}{2}\left(C_1(r)(t_1(r))^2+C_3(r)(t_3(r))^2\right).
\end{align*}
Now, using \eqref{new_eqn3} we obtain
\begin{align}
    H_2(r)\ge\frac{\mathcal{S}(r)}{(2r+1)^2},
    \label{p1_new-added_eqn2}
\end{align}
where 
\begin{align*}
    \mathcal{S}(r)&:=G_2'(0)(2r+1)^3+11\frac{G_2''(0)}{2}(2r+1)^2+(12G_2'(0)-90\frac{G_2''(0)}{2})(2r+1)+168\frac{G_2''(0)}{2}\\
    &=\frac{\pi^2}{3}(2r+1)^3-66\zeta(3)(2r+1)^2+(4\pi^2+540\zeta(3))(2r+1)-1008\zeta(3).
\end{align*}
Here, $\zeta(s)$ is the Riemann zeta function defined as $\zeta(s)=\sum_{n=1}^\infty\frac{1}{n^s}$, for $s\in \mathbb{C}$ with $Re(s)>1$.
\par
 For $r\ge4$, $\mathcal{S}'(r)$ has negative discriminant, so $\mathcal{S}'(r)>0$ for $r\ge4$ as $\mathcal{S}'(4)>0$. Therefore, $\mathcal{S}(r)$ is increasing for $r\ge4$. Hence, $\mathcal{S}(r)>0$ for all $r\ge 4$ as $\mathcal{S}(4)>0$. Consequently, from \eqref{p1_new-added_eqn2}, we obtain $H_2(r)>0$ for $r\ge 4$. This completes the proof of part (ii).
 \par Now, as shown in the proof of Theorem~\ref{p1_thm1.3}, part (iii) of the theorem readily follows using the fact that $a_{3}^{(r)}(n), b_{3}^{(r)}(n) \to \infty$ as $n \to \infty$ for all $r\geq 1$.
\par Finally, we prove part (iv). Using \eqref{p1_eqn5}, we obtain
\begin{align*}
    D_1^{(r)}&=2\psi\left(\frac{2}{2r+1}\right)-2\psi\left(\frac{1}{2r+1}\right)-2\ln2,\\
    D_2^{(r)}&=2\psi\left(\frac{4}{2r+1}\right)-2\psi\left(\frac{2}{2r+1}\right)-2\ln2,\\
    D_3^{(r)}&=2\psi\left(\frac{6}{2r+1}\right)-2\psi\left(\frac{3}{2r+1}\right)-2\ln2.
\end{align*}
Also, we know that $\psi(x) = -\gamma - \frac{1}{x} + O(x)$ as $x \to 0^+$. Thus, from \eqref{new-eqn-2} for sufficiently large $r$, we have
\begin{align*}
   \mathcal{D}(r)=&(3+2\ln 2)(2r+1)^2+(8r^2-10r-1)\left(2r+1+O\left(\frac{1}{r}\right)-2\ln2\right)\\
&+(-8r^2+28r-8)\left(\frac{2r+1}{2}+O\left(\frac{1}{r}\right)-2\ln2\right)\\
&+(4r^2-14r+10)\left(\frac{2r+1}{3}+O\left(\frac{1}{r}\right)-2\ln2\right).
 \end{align*}
 Therefore, $M_3(r)=\frac{(16r+13)(2r+1)^2}{6\mathcal{D}(r)}\to1$ as $r\to\infty$. This completes the proof of the Theorem.
\end{proof}
\begin{proof}[Proof of Theorem~\ref{p1_thm1.6}] We first prove part (i). By Theorems~\ref{p1_thm8} and\ref{p1_thm9}, we obtain
\begin{align*}
    M_4(r)&=\dfrac{(4r+2)(D_0^{(r)}+1) + (2r-3)(D_1^{(r)}-D_2^{(r)})}{(4r+2)D_0^{(r)}}.
\end{align*}
We aim to show that $M_4(r)>1$ for all $r \geq 1$. In the case $r=1$, direct calculation shows that  $M_4(1)=\frac{2\pi+2\sqrt{3}}{3\pi}>1$. Therefore, it is enough to consider $r \geq 2$. By using the fact that $\psi$ is strictly increasing on the positive real axis, the denominator of $M_4(r)$ is strictly positive. Thus, the problem reduces to establishing the following equivalent condition: \begin{align*}
    M_4(r)>1 &\iff (4r+2)+(2r-3)(D_1^{(r)}-D_2^{(r)})>0 \\ &\iff 2+(1-4t_1(r))(D_1^{(r)}-D_2^{(r)})>0.
\end{align*}
Again by Mean Value Theorem we have
\begin{align*}
    D_0^{(r)}-(D_1^{(r)}-D_2^{(r)})= \psi\left(\frac{1}{2}+2t_1(r)\right)-\psi\left(\frac{1}{2}+t_1(r)\right)=t_1(r)\psi'(\chi_2),
\end{align*}
for some $\chi_2\in \left(\frac{1}{2} + t_1(r),\frac{1}{2}+2t_1(r)\right)$, where the second equality is by using Mean Value Theorem. As $\psi'$ is decreasing function on positive real line and $\psi'(\frac{1}{2})=\frac{\pi^2}{2}$ and noting $t_1(r)\leq\frac{1}{5}$ for $r\geq2$, it follows that $t_1(r) \psi'(\chi_2)\le\frac{\psi'(1/2)}{5}=\frac{\pi^2}{10}<1$. Now,
\begin{align*}
    D_0^{(r)}>0 &\implies (D_1^{(r)}-D_2^{(r)})+t_1(r)\psi'(\chi_2)>0 \\
    &\implies (D_1^{(r)}-D_2^{(r)})>-1 \\
    &\implies (1-4t_1(r))(D_1^{(r)} - D_2^{(r)})>-1 \\
    &\implies 2+(1-4t_1(r))(D_1^{(r)}-D_2^{(r)})>0. 
\end{align*}
In the second-to-last step, we used $0<1-4t_1(r)<1$ for $r \geq 2$.
Therefore, $M_4(r)>1$. Finally, $a_1^{(r)}(n)\to \infty$ as $n\to \infty$ implies that $a_2^{(r)}(n)- a_1^{(r)}(n)\to \infty$ as $n\to \infty$.  
This completes the proof of part (i).
\par
Next, we prove part (ii) of the theorem. From Theorems~\ref{p1_thm9} and \ref{p1_thm10}, we obtain the following expression:
\begin{align*}
    M_5(r)=\dfrac{(2r+1)\left(2(2r+1)(D_0^{(r)}+1) + (2r-3)(D_1^{(r)}-D_2^{(r)})\right)}{\mathcal{D}(r)},
\end{align*}
where $\mathcal{D}(r)$ is as given in \eqref{new-eqn-2}. Using \eqref{p1_eqn6}, we have
\begin{align}
D_k^{(r)}=2\sum_{m\ge 0}\frac{(-1)^m}{kt_2(r)+m},\label{p1_eqn8}
\end{align}
where $t_2(r)=\frac{2}{2r+1} \text{ for } r\geq 1$. 
We divide the numerator and denominator of $M_5(r)$ by $(2r+1)^2$ and then using \eqref{p1_eqn5}, we obtain $M_5(r) = \frac{N(r)}{Q(r)}$, with $N(r)$ and $Q(r)$ as defined below.
\begin{align*}
N(r)&:=(2-2t_2(r))D_1^{(r)}-(1-2t_2(r))D_2^{(r)}+2\ln 2+2,\\
Q(r)&:=3+2\ln 2+\left(2-\frac{9t_2(r)}{2}+\frac{3(t_2(r))^2}{2}\right)D_1^{(r)}+(-2+9t_2(r)-6(t_2(r))^2)D_2^{(r)} \\ &+\left(1-\frac{9t_2(r)}{2}+\frac{9(t_2(r))^2}{2}\right)D_3^{(r)}.
\end{align*}
Hence, to prove that $M_5(r)>1$, we need to show that $N(r)>Q(r)$ with $Q(r) > 0$. Set $G_3(r):= N(r)-Q(r)$. Then, we have
\begin{align*}
G_3(r)=&\frac{t_2(r)(5-3t_2(r))}{2}D_1^{(r)}+(1-t_2(r))(1-6t_2(r))D_2^{(r)} \\ &-(1-3t_2(r))\left(1-\frac{3t_2(r)}{2}\right)D_3^{(r)}-1.
\end{align*}
Applying \eqref{p1_eqn8}, we obtain
\begin{align*}
G_3(r)=\frac{1}{3t_2(r)}+2\sum_{m\ge1}(-1)^mc_r(m), 
\end{align*}
 where  \begin{align*}
     c_r(m)=\frac{t_2(r)\left((1-2t_2(r)+3(t_2(r))^2)m+t_2(r)(1+3t_2(r))\right)}{(m+t_2(r))(m+2t_2(r))(m+3t_2(r))}.
 \end{align*}
  Since $c_r(m)$ is a positive and decreasing sequence of $m$ for each fixed $t_2(r)$, and $c_r(m)\to 0$ as $ m\to\infty$, therefore, the alternating series satisfies $\sum_{m\ge 1}(-1)^mc_r(m)\ge-c_r(1)$.
Hence,
\begin{align*}
G_3(r)\ge\frac{1}{3t_2(r)}-\frac{2t_2(r)(1-t_2(r)+6(t_2(r))^2)}{(1+t_2(r))(1+2t_2(r))(1+3t_2(r))}.
\end{align*}
Showing $G_3(r)>0$ is enough to showing that $6(t_2(r))^2(1-t_2(r)+6(t_2(r))^2)<1+6t_2(r)+11(t_2(r))^2+6(t_2(r))^3$. Equivalently, we need to show that $1+6t_2(r)+5(t_2(r))^2+12(t_2(r))^3-36(t_2(r))^4>0$, which is true for $0<t_2(r)\leq \frac{2}{3}$.
Again, using \eqref{p1_eqn8}, we find that
\begin{align*}
Q(r)=2\ln2+\frac{8}{3t_2(r)}+2\sum_{m\ge1}(-1)^md_r(m),
\end{align*} 
 where 
 \begin{align*}d_r(m)=\frac{m^2+(5t_2(r)-3(t_2(r))^3)m+(8(t_2(r))^2-9(t_2(r))^3)}{(m+t_2(r))(m+2t_2(r))(m+3t_2(r))}.
 \end{align*}
It is easy to verify that $d_r(m)$ forms a positive and decreasing sequence of $m$ for each fixed $t_2(r)$ and $d_r(m)\to 0$ as $m\to \infty$. Therefore, $\sum_{m\ge 1}(-1)^m d_r(m)\ge-d_r(1)$, and $d_r(1)<1$ for $0<t_2(r)\leq \frac{2}{3}$. It follows that $Q(r) > 2\ln 2 + \frac{8}{3t_2(r)}-2\ge2\ln2+4-2>0$. This completes the proof of $M_5(r)>1$ for all $r\geq 1$. Finally, $a_3^{(r)}(n)\to \infty$ as $n\to \infty$ implies that $a_2^{(r)}(n)- a_3^{(r)}(n)\to \infty$ as $n\to \infty$.  
This completes the proof of part (ii).
\par Finally, we prove part (iii) of the theorem. Employing Theorems~\ref{p1_thm8} and \ref{p1_thm10}, we have 
\begin{align}\label{eqn-new-12}
M_6(r)
= \frac{\mathcal{D}(r)}{2(2r+1)^2 D_0^{(r)}}
\end{align}
where $\mathcal{D}(r)$ is as given in \eqref{new-eqn-2}. We replace $\mathcal{D}(r)$ using \eqref{new-eqn-2} and $D_0^{(r)}$ using \eqref{p1_eqn5} in \eqref{eqn-new-12}, and then multiply both the numerator and the denominator by $(t_1(r))^2$ to deduce the following expression
\begin{align*}
    M_6(r)&= \frac{(3+2\ln 2)
+ (2-9t_1(r)+6(t_1(r))^2)D_1^{(r)}
+ (-2+18t_1(r)-24(t_1(r))^2)D_2^{(r)}}
{2\ln 2+D_1^{(r)}} \\
&\qquad + \frac{(1-9t_1(r)+18(t_1(r))^2)D_3^{(r)}}
{2\ln 2+D_1^{(r)}}.
\end{align*}
Note that the denominator of $M_6(r)$ is always positive, as $\psi$ is strictly increasing on the positive real line. Now, we have $M_6(r)>1$ if and only if 
\begin{align*}
    & 3+(1-9t_1(r)+6(t_1(r))^2)D_1^{(r)}+(-2+18t_1(r)-24(t_1(r))^2)D_2^{(r)}\\
    &+(1-9t_1(r)+18(t_1(r))^2)D_3^{(r)}>0.
\end{align*}
Let
\begin{align*}
    G_4(r)&:=3+(1-9t_1(r)+6(t_1(r))^2)D_1^{(r)}+(-2+18t_1(r)-24(t_1(r))^2)D_2^{(r)} \\ &+(1-9t_1(r)+18(t_1(r))^2)D_3^{(r)}.
\end{align*}
Now, we prove that $G_4(r)>0$. Using \eqref{p1_eqn6}, we obtain
\begin{align}\notag
    G_4(r)&=3+2\sum_{k\ge0}(-1)^k\left(\frac{1-9t_1(r)+6(t_1(r))^2}{2t_1(r)+k}+\frac{-2+18t_1(r)-24(t_1(r))^2}{4t_1(r)+k}\right.\\
&\hspace{1 cm}\left.+  \frac{1-9t_1(r)+18(t_1(r))^2}{6t_1(r)+k}\right)
\notag\\
    &=\frac{1}{3t_1(r)}+\left(16(t_1(r))^2(1-9t_1(r))Z_A-48(t_1(r))^3Z_B\right),\label{p1_new-added_eqn3}
\end{align}
where
\begin{align*}
&Z_A:=\sum_{k\ge1}\frac{(-1)^k}{\rho_r(k)},\quad Z_B:=\sum_{k\ge1}\frac{(-1)^k\,k}{\rho_r(k)};\\
&\rho_r(k):=\left(2t_1(r)+k)(4t_1(r)+k)(6t_1(r)+k\right).
\end{align*}
It is easy to verify that both series $Z_A \text{ and } Z_B$ are absolutely convergent, since $\frac{1}{\rho_r(k)}>0$ and $\frac{k}{\rho_r(k)}>0$ for all $k\geq 1$, for each fixed $r$ and these are decreasing sequences of $k$ and the terms are of $O(k^{-3})$ and $O(k^{-2})$, respectively.
Observe that
\begin{align*}
    Z_A&=-\frac{1}{\rho_r(1)}+\sum_{k\ge2}\left(\frac{1}{\rho_r(k)}-\frac{1}{\rho_r(k+1)}\right)\ge -\frac{1}{\rho_r(1)},\\
    Z_A&=\sum_{k\ge1}\left(-\frac{1}{\rho_r(k)}+\frac{1}{\rho_r(k+1)}\right)\le 0.
\end{align*}
Hence,
\begin{align*}
    \frac{-1}{(2t_1(r)+1)(4t_1(r)+1)(6t_1(r)+1)}\leq Z_A\leq 0.
\end{align*}
Similarly, we have $Z_B\leq 0$, which implies that $-48(t_1(r))^3 Z_B\geq 0$. Since the denominator of $M_6(r)$ is strictly positive, it remains to prove that $G_4(r)>0$. We consider two cases according to the value of $t_1(r)$.\\
\textbf{Case 1:} Suppose that $\frac{1}{9}<t_1(r)\leq\frac{1}{3}$ (equivalently, $r\leq3$). Then $1-9t_1(r) < 0$. Since $Z_A \leq 0$, we obtain  
$8(t_1(r))^2(1-9t_1(r))Z_A \geq 0$. Consequently, we have $G_4 \geq \frac{1}{3t_1(r)} > 0$.\\
\textbf{Case 2:} Suppose that $0 < t_1(r) \leq \frac{1}{9}$ (equivalently, $r \geq 4$). In this case, $1-9t_1(r) \in [0, 1)$ and $(2t_1(r)+1)(4t_1(r)+1)(6t_1(r)+1) > 1$. Consequently, we have 
\begin{align*}
16(t_1(r))^2(1-9t_1(r))Z_A\geq\frac{-16(t_1(r))^2}{(2t_1(r)+1)(4t_1(r)+1)(6t_1(r)+1)}>-16(t_1(r))^2.
\end{align*}
Thus, by \eqref{p1_new-added_eqn3}, we obtain
\begin{align*}
G_4(r)>\frac{1}{3t_1(r)}-16(t_1(r))^2-48(t_1(r))^3 Z_B\geq\frac{1}{3t_1(r)}-16(t_1(r))^2\geq 3-\frac{16}{81}>0
\end{align*}
Therefore, $G_4(r)>0$ in both cases, which implies that $M_6(r)>1$. Hence, $M_6(r)>1$ holds for every $r\ge1$. Finally, as shown in the proof of Theorem~\ref{p1_thm1.3}, $a_1^{(r)}(n)\to \infty$ as $n\to \infty$ implies that $a_3^{(r)}(n)- a_1^{(r)}(n)\to \infty$ as $n\to \infty$.  
This completes the proof of part (iii).
\end{proof}
\begin{remark} In part (i) of Theorem~\ref{p1_thm1.6}, we established that $a_2^{(r)}(n) > a_1^{(r)}(n)$ for all sufficiently large $n$, and that $a_2^{(r)}(n) - a_1^{(r)}(n) \to \infty$ as $n \to \infty$. We now also provide a simple combinatorial justification for it. For any partition $\lambda$ of $n$, the number of $1$-hooks is equal to the number of distinct parts of $\lambda$. In an $A_r$-type partition, each part must occur at least twice, which guarantees that for each distinct part, there exists at least one $2$-hook. Therefore, for every $A_r$-type partition of $n$, the number of $2$-hooks is at least the number of $1$-hooks, so $a_2^{(r)}(n) \geq a_1^{(r)}(n)$ for all $n \geq 1$.\par To show that the inequality is strict for sufficiently large $n$, let $n \geq 2r + 7$ and define $\mathcal{d}:= n - (2r+7)$. By the division algorithm, let $\mathcal{d} = 6m + s$ for integers $m \geq 0$ and $0 \leq s < 6$. For every integer $m_0$ with $0 \leq m_0 \leq m$, consider the partition $\lambda = (3^{2(m_0+1)}, 1^{2r + s + 6(m - m_0) + 1})$. This is clearly an $A_r$-type partition of $n$. In each such partition, the number of $2$-hooks exceeds the number of $1$-hooks by exactly one. As $m_0$ ranges from $0$ to $m$, there are $m+1$ such partitions, each contributing one additional $2$-hook beyond the number of $1$-hooks. Thus, $a_2^{(r)}(n) > a_1^{(r)}(n)$ for all $n \geq 2r + 7$. Moreover, as $n \to \infty$ implies  $m \to \infty$, thus the difference $a_2^{(r)}(n) - a_1^{(r)}(n) \to \infty $, as required.
\end{remark}

\begin{proof}[Proof of Theorem~\ref{p1_thm1.7}]
Following similarly, we show that $M_7(r), M_8(r), M_9(r) >1$ to establish the biases. From Theorems~\ref{p1_thm8}, \ref{p1_thm9}, and \ref{p1_thm10}, we obtain
\begin{align*}
   M_7(r) &= \frac{3r+2}{2(r+1)} = 1+ \frac{r}{2r+2}>1,\\
   M_8(r) &= \frac{18r+12}{16r+13} = 1+ \frac{2r-1}{16r+13}>1, \\
    M_9(r) &= \frac{16r+13}{12(r+1)} = 1+ \frac{4r+1}{12(r+1)}>1.
\end{align*}
Furthermore, $b_{\ell}^{(r)}(n)\to \infty$ as $n \to \infty$ for $\ell=1,3$. Hence, as shown in the proof of Theorem~\ref{p1_thm1.3}, we have $b_2^{(r)}(n)-b_{\ell}^{(r)}(n) \to\infty$ as $n\to\infty$ for $\ell=1,3$.
Similarly, $b_3^{(r)}(n)-b_1^{(r)}(n)\to\infty$ as $n\to\infty$. This completes the proof of the theorem. 
\end{proof}
\begin{proof}[Proof of Theorem~\ref{p1_newthm}] From Theorems~\ref{p1_thm8}, \ref{p1_thm9}, and \ref{p1_thm10} we have 
\begin{align*}
    \lim_{n\to\infty}\frac{b_1^{(r)}(n)}{b_1^{(r+1)}(n)} &=   \lim_{n \to \infty }e^{2\pi\sqrt{n}\left(\sqrt{\frac{r+1}{6(2r+1)}}-\sqrt{\frac{r+2}{6(2r+3)}}\right)}\frac{(r+1)(2r+3)}{(2r+1)(r+2)}  ,\\
    \lim_{n\to\infty}\frac{b_2^{(r)}(n)}{b_2^{(r+1)}(n)} &=   \lim_{n \to \infty }e^{2\pi\sqrt{n}\left(\sqrt{\frac{r+1}{6(2r+1)}}-\sqrt{\frac{r+2}{6(2r+3)}}\right)}\frac{(3r+2)(4r+6)}{(3r+5)(4r+2)}  ,\\
    \lim_{n\to\infty}\frac{b_3^{(r)}(n)}{b_3^{(r+1)}(n)} &=  \lim_{n \to \infty }e^{2\pi\sqrt{n}\left(\sqrt{\frac{r+1}{6(2r+1)}}-\sqrt{\frac{r+2}{6(2r+3)}}\right)}\frac{(16r+13)(24r+36)}{(16r+29)(24r+12)}.
\end{align*}
 Observe that $\sqrt{\frac{r+1}{6(2r+1)}} - \sqrt{\frac{r+2}{6(2r+3)}} > 0$ for every $r\geq 1$. Consequently, the exponential factor tends to infinity as $n\to\infty$. Moreover, each of the rational factors $\frac{(r+1)(2r+3)}{(2r+1)(r+2)}$, $\frac{(3r+2)(4r+6)}{(3r+5)(4r+2)}$, and $\frac{(16r+13)(24r+36)}{(16r+29)(24r+12)}$ is positive.  Hence, for every $r\geq 1$,
$$
\lim_{n\to\infty}\frac{b_1^{(r)}(n)}{b_1^{(r+1)}(n)}=\lim_{n\to\infty}\frac{b_2^{(r)}(n)}{b_2^{(r+1)}(n)}=\lim_{n\to\infty}\frac{b_3^{(r)}(n)}{b_3^{(r+1)}(n)}=\infty.
$$
Thus, all three limits establish the required biases. Furthermore, $b_{\ell}^{(r)}(n)\to \infty$ as $n \to \infty$ for $\ell=1, 2, 3$ and every $r\geq 1$. Hence, as shown in the proof of Theorem~\ref{p1_thm1.3}, $b_{\ell}^{(r)}(n)-b_{\ell}^{(r+1)}(n) \to\infty$ as $n\to\infty$ for $\ell=1,2, 3$. 
 This completes the proof of the theorem. 
 \end{proof}
 
\section*{Acknowledgement}
We thank Ajit Singh and Gurinder Singh for their careful reading of the manuscript and their valuable suggestions.

\bibliographystyle{amsplain}
		\bibliography{ref_biblio.bib}

@article {MacMahon_Generalization_by_Andrews,
    AUTHOR = {Andrews, George E.},
     TITLE = {A generalization of a partition theorem of {M}ac{M}ahon},
   JOURNAL = {J. Combinatorial Theory},
  FJOURNAL = {Journal of Combinatorial Theory},
    VOLUME = {3},
      YEAR = {1967},
     PAGES = {100--101},
  MRNUMBER = {211891},
}

@article {Ono_Singh,
    AUTHOR = {Amdeberhan, Tewodros and Andrews, George E. and Ono, Ken and Singh, Ajit},
     TITLE = {Hook lengths in self-conjugate partitions},
   JOURNAL = {Proc. Amer. Math. Soc. Ser. B},
  FJOURNAL = {Proceedings of the American Mathematical Society. Series B},
    VOLUME = {11},
      YEAR = {2024},
     PAGES = {345--357},
  MRNUMBER = {4767988},
MRREVIEWER = {Olivier\ Mallet},
}

@article {Ballantine_2023,
    AUTHOR = {Ballantine, Cristina and Burson, Hannah E. and Craig, William
              and Folsom, Amanda and Wen, Boya},
     TITLE = {Hook length biases and general linear partition inequalities},
   JOURNAL = {Res. Math. Sci.},
  FJOURNAL = {Research in the Mathematical Sciences},
    VOLUME = {10},
      YEAR = {2023},
    NUMBER = {4},
     PAGES = {Paper No. 41, 36},
  MRNUMBER = {4648745},
}

@article {BMS,
    AUTHOR = {Barman, Rupam and Mahanta, Pankaj Jyoti and Singh, Gurinder},
     TITLE = {Hook length inequalities for {$t$}-regular partitions in the {$t$}-aspects: {P}art {II}},
   JOURNAL = {Res. Number Theory},
  FJOURNAL = {Research in Number Theory},
    VOLUME = {11},
      YEAR = {2025},
    NUMBER = {4},
     PAGES = {Paper No. 84, 15},
  MRNUMBER = {4956722},
MRREVIEWER = {Huan\ Xiong},
}

@article {Ingham,
    AUTHOR = {Bringmann, Kathrin and Jennings-Shaffer, Chris and Mahlburg, Karl},
     TITLE = {On a {T}auberian theorem of {I}ngham and {E}uler-{M}aclaurin summation},
   JOURNAL = {Ramanujan J.},
  FJOURNAL = {Ramanujan Journal. An International Journal Devoted to the
              Areas of Mathematics Influenced by Ramanujan},
    VOLUME = {61},
      YEAR = {2023},
    NUMBER = {1},
     PAGES = {55--86},
  MRNUMBER = {4584516},
MRREVIEWER = {Driss\ Essouabri},
}

@article {Cossaboom,
    AUTHOR = {Cossaboom, Catherine H.},
     TITLE = {Hook length biases for self-conjugate partitions and
              partitions with distinct odd parts},
   JOURNAL = {J. Number Theory},
  FJOURNAL = {Journal of Number Theory},
    VOLUME = {277},
      YEAR = {2025},
     PAGES = {290--324},
  MRNUMBER = {4908737},
MRREVIEWER = {Huan\ Xiong},
}

@article {Craig,
    AUTHOR = {Craig, William and Locus Dawsey, Madeline and Han, Guo-Niu},
     TITLE = {Inequalities and asymptotics for hook numbers in restricted
              partitions},
   JOURNAL = {J. Combin. Theory Ser. A},
  FJOURNAL = {Journal of Combinatorial Theory. Series A},
    VOLUME = {221},
      YEAR = {2026},
     PAGES = {Paper No. 106176, 35},
  MRNUMBER = {5036914},
}

@article {Craig2,
    AUTHOR = {Craig, William and Ono, Ken and Singh, Ajit},
     TITLE = {Distribution of hooks in self-conjugate partitions},
   JOURNAL = {Discrete Math.},
  FJOURNAL = {Discrete Mathematics},
    VOLUME = {348},
      YEAR = {2025},
    NUMBER = {10},
     PAGES = {Paper No. 114563, 13},
  MRNUMBER = {4904414},
}

@article {Craig3,
    AUTHOR = {Craig, William},
     TITLE = {On the number of parts in congruence classes for partitions
              into distinct parts},
   JOURNAL = {Res. Number Theory},
  FJOURNAL = {Research in Number Theory},
    VOLUME = {8},
      YEAR = {2022},
    NUMBER = {3},
     PAGES = {Paper No. 52, 24},
  MRNUMBER = {4456287},
}

@article {Sellers,
    AUTHOR = {Fu, Shishuo and Sellers, James A.},
     TITLE = {Bijective proofs of partition identities of {M}ac{M}ahon,
              {A}ndrews, and {S}ubbarao},
   JOURNAL = {Electron. J. Combin.},
  FJOURNAL = {Electronic Journal of Combinatorics},
    VOLUME = {21},
      YEAR = {2014},
    NUMBER = {2},
     PAGES = {Paper 2.41, 9},
  MRNUMBER = {3244807},
}

@article {Griffin,
    AUTHOR = {Griffin, Michael and Ono, Ken and Tsai, Wei-Lun},
     TITLE = {Distributions of hook lengths in integer partitions},
   JOURNAL = {Proc. Amer. Math. Soc. Ser. B},
  FJOURNAL = {Proceedings of the American Mathematical Society. Series B},
    VOLUME = {11},
      YEAR = {2024},
     PAGES = {422--435},
  MRNUMBER = {4797102},
}

@article {Hardy_Ramanujan,
    AUTHOR = {Hardy, G. H. and Ramanujan, S.},
     TITLE = {Asymptotic {F}ormulae in {C}ombinatory {A}nalysis},
   JOURNAL = {Proc. London Math. Soc. (2)},
  FJOURNAL = {Proceedings of the London Mathematical Society. Second Series},
    VOLUME = {17},
      YEAR = {1918},
     PAGES = {75--115},
  MRNUMBER = {1575586},
}

@article {Kim1,
    AUTHOR = {Kim, Eunmi},
     TITLE = {Inequalities and asymptotics for hook lengths in
              {$\ell$}-regular partitions and {$\ell$}-distinct partitions},
   JOURNAL = {J. Math. Anal. Appl.},
  FJOURNAL = {Journal of Mathematical Analysis and Applications},
    VOLUME = {550},
      YEAR = {2025},
    NUMBER = {2},
     PAGES = {Paper No. 129599, 19},
  MRNUMBER = {4897658},
}

@article {Rademacher,
    AUTHOR = {Rademacher, Hans},
     TITLE = {On the {P}artition {F}unction p(n)},
   JOURNAL = {Proc. London Math. Soc. (2)},
  FJOURNAL = {Proceedings of the London Mathematical Society. Second Series},
    VOLUME = {43},
      YEAR = {1937},
    NUMBER = {4},
     PAGES = {241--254},
  MRNUMBER = {1575213},
}

@article {Ngo_Rhoades,
    AUTHOR = {Ngo, Hieu T. and Rhoades, Robert C.},
     TITLE = {Integer partitions, probabilities and quantum modular forms},
   JOURNAL = {Res. Math. Sci.},
  FJOURNAL = {Research in the Mathematical Sciences},
    VOLUME = {4},
      YEAR = {2017},
     PAGES = {Paper No. 17, 36},
  MRNUMBER = {3721075},
}

@article {Jackson,
    AUTHOR = {Jackson, Faye and Otgonbayar, Misheel},
     TITLE = {Unexpected biases between congruence classes for parts in
              {$k$}-indivisible partitions},
   JOURNAL = {J. Number Theory},
  FJOURNAL = {Journal of Number Theory},
    VOLUME = {248},
      YEAR = {2023},
     PAGES = {310--342},
  MRNUMBER = {4562070},
}

@book {MacMahon,
    AUTHOR = {MacMahon, Percy A.},
     TITLE = {Combinatory analysis},
      NOTE = {Two volumes (bound as one)},
 PUBLISHER = {Chelsea Publishing Co., New York},
      YEAR = {1960},
     PAGES = {xix+302+xix+340},
  MRNUMBER = {141605},
}

@incollection {Nekrasov_Okounkov,
    AUTHOR = {Nekrasov, Nikita A. and Okounkov, Andrei},
     TITLE = {Seiberg-{W}itten theory and random partitions},
 BOOKTITLE = {The unity of mathematics},
    SERIES = {Progr. Math.},
    VOLUME = {244},
     PAGES = {525--596},
 PUBLISHER = {Birkh\"auser Boston, Boston, MA},
      YEAR = {2006},
  MRNUMBER = {2181816},
}

@article {Han_2010,
	AUTHOR = {Han, Guo-Niu},
	TITLE = {The {N}ekrasov-{O}kounkov hook length formula: refinement, elementary proof, extension and applications},
	JOURNAL = {Ann. Inst. Fourier},
	VOLUME = {60},
	YEAR = {2010},
	NUMBER = {1},
	PAGES = {1--29},
}

@incollection {Han_2016,
    AUTHOR = {Han, Guo-Niu and Xiong, Huan},
     TITLE = {New hook-content formulas for strict partitions (extended
              abstract)},
 BOOKTITLE = {28th {I}nternational {C}onference on {F}ormal {P}ower {S}eries
              and {A}lgebraic {C}ombinatorics ({FPSAC} 2016)},
    SERIES = {Discrete Math. Theor. Comput. Sci. Proc.},
    VOLUME = {BC},
     PAGES = {635--645},
 PUBLISHER = {Assoc. Discrete Math. Theor. Comput. Sci., Nancy},
      YEAR = {[2016] \copyright 2016},
  MRNUMBER = {4111324},
}

@article {Han_2017,
    AUTHOR = {Han, Guo-Niu and Xiong, Huan},
     TITLE = {New hook-content formulas for strict partitions},
   JOURNAL = {J. Algebraic Combin.},
  FJOURNAL = {Journal of Algebraic Combinatorics. An International Journal},
    VOLUME = {45},
      YEAR = {2017},
    NUMBER = {4},
     PAGES = {1001--1019},
  MRNUMBER = {3641975},
}

@article {Petreolle,
    AUTHOR = {P\'etr\'eolle, Mathias},
     TITLE = {Characterization of cyclically fully commutative elements in
              finite and affine {C}oxeter groups},
   JOURNAL = {European J. Combin.},
  FJOURNAL = {European Journal of Combinatorics},
    VOLUME = {61},
      YEAR = {2017},
     PAGES = {106--132},
  MRNUMBER = {3588712},
}

@article {Singh1Barman,
    AUTHOR = {Singh, Gurinder and Barman, Rupam},
     TITLE = {Hook length biases in ordinary and {$t$}-regular partitions},
   JOURNAL = {J. Number Theory},
  FJOURNAL = {Journal of Number Theory},
    VOLUME = {264},
      YEAR = {2024},
     PAGES = {41--58},
  MRNUMBER = {4767463},
}

@article {Singh2Barman,
    AUTHOR = {Singh, Gurinder and Barman, Rupam},
     TITLE = {Hook length inequalities for {$t$}-regular partitions in the {$t$}-aspect},
   JOURNAL = {Canad. Math. Bull.},
  FJOURNAL = {Canadian Mathematical Bulletin. Bulletin Canadien de
              Math\'ematiques},
    VOLUME = {69},
      YEAR = {2026},
    NUMBER = {1},
     PAGES = {77--88},
  MRNUMBER = {5021592},
}

@article {Bing-He,
	AUTHOR = {He, Bing and Liu, Shuming},
	TITLE = {A proof of a conjecture of {S}ingh and {B}arman on hook
	length},
	JOURNAL = {J. Number Theory},
	FJOURNAL = {Journal of Number Theory},
	VOLUME = {278},
	YEAR = {2026},
	PAGES = {353--379},
     MRNUMBER = {4919847},
}

@article {Zang,
	AUTHOR = {Qu, Wenxia and Zang, Wenston J. T.},
	TITLE = {On the hook length biases of the 2- and 3-regular partitions},
	JOURNAL = {J. Number Theory},
	FJOURNAL = {Journal of Number Theory},
	VOLUME = {280},
	YEAR = {2026},
	PAGES = {455--480},
    MRNUMBER = {4966416},
}

@article {Yee,
	AUTHOR = {Cho, Hyunsoo and Kim, Byungchan and Kim, Eunmi and Yee, Ae Ja},
	TITLE = {Inequalities on the total number of {$t$}-hooks of partitions},
	JOURNAL = {European J. Combin.},
	FJOURNAL = {European Journal of Combinatorics},
	VOLUME = {136},
	YEAR = {2026},
	PAGES = {Paper No. 104390, 29},
     MRNUMBER = {5062381},
}

@article {Wright,
    AUTHOR = {Wright, E. M.},
     TITLE = {Stacks. {II}},
   JOURNAL = {Quart. J. Math. Oxford Ser. (2)},
  FJOURNAL = {The Quarterly Journal of Mathematics. Oxford. Second Series},
    VOLUME = {22},
      YEAR = {1971},
     PAGES = {107--116},
  MRNUMBER = {282940},
}

\end{document}